\documentclass[a4paper, 11pt,reqno]{amsart}
\usepackage{graphicx}
\usepackage{amsmath, amsfonts, amssymb, amsthm}
\usepackage{xcolor}
\usepackage[margin=0.75in]{geometry}

\usepackage{booktabs} 

\newcommand{\Z}{\mathbb{Z}}
\newcommand{\TT}{\mathcal T}
\newcommand{\RR}{\mathcal R}
\newcommand{\R}{\mathbb R}
\newcommand{\abs}[1]{\left|#1\right|} 
\newcommand{\ceiling}[1]{\left\lceil #1 \right\rceil}
\newcommand{\floor}[1]{\left\lfloor #1 \right\rfloor}
\newcommand{\dv}{\mathrm{d}}

\newcommand{\cI}{\mathcal{I}}
\newcommand{\cR}{\mathcal{R}}
\newcommand{\cD}{\mathcal{D}}

\DeclareMathOperator{\Var}{Var}
\DeclareMathOperator{\Int}{Int}

\newtheorem{theorem}{Theorem}[section]
\newtheorem{lemma}[theorem]{Lemma}
\newtheorem{corollary}[theorem]{Corollary}

\theoremstyle{remark}

\usepackage{subfigure}  
\usepackage{float}
\usepackage[]{caption}
\newcommand{\sdfrac}[2]{\mbox{\small$\displaystyle\frac{#1}{#2}$}}
\newcommand{\fdfrac}[2]{\mbox{\footnotesize$\displaystyle\frac{#1}{#2}$}}
\newcommand{\ldfrac}[2]{\mbox{\large$\frac{#1}{#2}$}}
\newcommand{\Ldfrac}[2]{\mbox{\Large$\frac{#1}{#2}$}}

\definecolor{orange}{rgb}{1,0.5,0}
\definecolor{Red}{rgb}{.795,0.015,0.017}
\definecolor{Ggreen}{rgb}{0.,0.675,0.0128}
\definecolor{Bblue}{rgb}{0.16,.32,0.91}

\usepackage[hyphens]{url}

\usepackage[backref=page]{hyperref}
\hypersetup{
	colorlinks = true,
	linkcolor={Red},
	urlcolor={blue},
	citecolor={Ggreen},    
	urlcolor = {blue},
	citebordercolor = {0.33 .58 0.33},
	linkbordercolor = {0.99 .28 0.23},
	breaklinks=true
}
\usepackage{cite}

\begin{document}

	\title{Distribution of angles to lattice points seen from a fast moving observer}
	\author{Jack Anderson, Florin P. Boca, Cristian Cobeli, Alexandru Zaharescu}
	
	
	\address[Jack Anderson]{Department of Mathematics, University of Illinois at Urbana-Champaign, Urbana, IL 61801, USA.}
	\email{jacka4@illinois.edu}
	
	\address[Florin P. Boca]{Department of Mathematics, University of Illinois at Urbana-Champaign, Urbana, IL 61801, USA.}
	\email{fboca@illinois.edu}
	
	\address[Cristian Cobeli]{"Simion Stoilow" Institute of Mathematics of the Romanian Academy,~21 Calea Grivitei Street, P. O. Box 1-764, Bucharest 014700, Romania}
	\email{cristian.cobeli@imar.ro}
	
	\address[Alexandru Zaharescu]{Department of Mathematics,University of Illinois at Urbana-Champaign, Urbana, IL 61801, USA,
			and 
			"Simion Stoilow" Institute of Mathematics of the Romanian Academy,~21 
			Calea Grivitei 
			Street, P. O. Box 1-764, Bucharest 014700, Romania}
		\email{zaharesc@illinois.edu}

		\subjclass[2020]{Primary 11P21.
			Secondary: 11K99,   11B99.}
		
		\thanks{Key words and phrases: Lattice points,  
			local spacing statistics,
			gap distribution of angles,
			exponential sums.}
		
		\begin{abstract}
			We consider a square expanding with constant speed seen from an observer moving away with constant acceleration and
			study the distribution of angles between
			rays from the observer towards the lattice points in the square.
			We prove the existence of the gap 
			distribution as time tends to infinity and
			provide explicit formulas for the corresponding density function.
		\end{abstract}
		
		\maketitle
		
		\bigskip
		
		\noindent
		
		\bigskip
		
		\section{Introduction}\label{S:Introduction}
		
			The spacing statistics of the angular distributions of Euclidean or hyperbolic lattice points or of cut-and-project quasicrystals was thoroughly investigated in many works, revealing interesting connections with number theory and ergodic theory. See \cite{BGHJ2014, BCZ1, BG2009, BPPZ, BPZ, CKLW, EMV, EM, Ham, KK, Ma2, MS1, MS2, MV, RS} for a far from exhaustive list of papers on these topics. A basic example arises by considering all lattice points in the dilated triangle~$J\TT$, where
			\begin{equation*}
				\TT=\{ (x,y):0< y\leq x\leq 1\}, \qquad J\rightarrow \infty. \\[4pt] 
			\end{equation*}
			The slopes of the lines $OP$, $O=(0,0)$, $P\in J\TT \cap \Z^2$ with coprime coordinates, represent exactly the Farey points of height $J$ in the interval $[0,1]$. Geometrically, we look at the points $OP$ that are visible to an observer located at the origin $O$. Certainly, there is a subtle difference between considering all visible points versus all points in $J\TT$. The spacing statistics of Farey points and generalizations have been investigated in depth in \cite{Ath2016, AC, ABCZ, BCZ2, BZ, Hal, Hee, Lut, Ma1, PP, Zha}.
			
			In the present paper we consider a related new model, where the observer is no longer fixed.
		More precisely, we consider the lattice points in a square expanding with constant speed, which are viewed from an observer moving far away with constant acceleration.
		The analysis of this new type of counting problems requires different techniques that we develop in this paper.
		\begin{figure}[ht]
			\centering    \includegraphics[angle=0,width=0.95\textwidth]{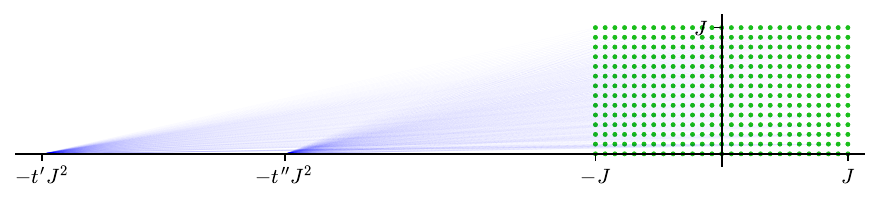}
			\vspace*{-10pt}
			\caption{Two points of view 
				$PoV=(-t'J^2,0)$ and $PoV=(-t''J^2,0)$, for $0<t''<t'$, of an observer who, coming from afar, sees the lattice points in the upper half of the square $[-J,J]^2$.
			}
			\label{FigurePoVs}
		\end{figure}
			
			To illustrate the different features of the angular distribution while keeping calculations more accessible, 
			we only consider here the situation where the observer is located at the point $P_{t,J}=(-tJ^2,0)$ with $t>0$ fixed.
			(Here $t$ is just a parameter controlling the position of the observer and it is not related to time.)
			As we are interested in the distribution of angles, for symmetry reasons it will be sufficient to consider the case 
			of angles between the rays sent by the observer only towards the points above the $x$-axis. 
			In Figure~\ref{FigurePoVs}, these rays sent towards the lattice points can be seen from two different perspectives 
			as we approach the target square.
			Thus the triangles $J\TT$ are replaced by rectangles $J\RR$, $\RR=[-1,1]\times [0,1]$, 
			so that $\lvert P_{t,J} P\rvert \sim \frac{t}{2} \# \RR_J$ as time $J\rightarrow \infty$, 
			where we set $\RR_J =J\RR$ with 
			\begin{equation*}
				N=N_J:=\#\RR_J =(2J+1)(J+1)=2J^2+O(J).   
			\end{equation*}
			We can order the angles $\angle (PP_{t,J} O)$ as 
			$0 = \alpha_1\leq\alpha_2\leq\dots
			\leq \alpha_{N} = \alpha_{\max}$. 
			First, we calculate the average $\Delta_{av}=\frac{1}{N_J}\alpha_{\max}$. We have
			\begin{equation*}
				\tan(\alpha_{\max}) = \frac{J}{tJ^2-J} = \frac{1}{tJ} 
				+ O\left(\sdfrac{1}{J^2}\right),
			\end{equation*}
			so,
			\begin{equation*}
				\alpha_{\max} = \frac{1}{tJ} + O\left(\sdfrac{1}{J^2}\right),
			\end{equation*}
			hence,
			\begin{equation}\label{E:AverageAngle}
				\Delta_{av} = \frac{1}{2tJ^3} + O\left(\sdfrac{1}{J^4}\right).
			\end{equation}
			The elements of the finite sequence $S_{t,J}:=\{ \alpha_j\}_{j=1}^{N}$ are plainly seen to be uniformly distributed, in the sense that
			\begin{equation*}
				\lim\limits_{J\to\infty} \frac{1}{N_J} 
				\# \big\{ P\in \RR_J : \angle (PP_{t,J} O) \in [\alpha \Delta_{av}, \beta \Delta_{av}]\big\} 
				=\beta -\alpha ,\ \ \text{ for } 0\leq \alpha \leq \beta \leq 1.
			\end{equation*}
			
			The \emph{gap distribution function} associated to $S_{t,J}$ is defined as the proportion of angles greater than or equal to~$\lambda$ times the average, that is,
			\begin{equation*}
				G_{t,J}(\lambda) := \frac{1}{N_J}
				\#\big\{j\colon \alpha_{j+1}-\alpha_j \geq \lambda\Delta_{av}\big\},
				\ \ \text{ for } \lambda >0.
			\end{equation*}
			The aim of this paper is to show the existence of the limiting gap distribution function
			\begin{equation*}
				G_t(\lambda):=\lim\limits_{J\to\infty} G_{t,J} (\lambda),    
			\end{equation*}
			and to describe how to compute it.
			
			It is natural to ask what is the behavior of the gap distribution when the observer is located at the point $(-tJ^\alpha,0)$ for some fixed $\alpha >0$. The situation $\alpha >2$ corresponds to $t\rightarrow \infty$ in Theorem \ref{T:GLambdaClosedForm}, whence $G_t=\delta_0$.
			The situation $0<\alpha <2$ corresponds to $t\rightarrow 0$ in Theorems \ref{T:GLambdaAllt} and \ref{T:GLambdaAlltConstants}, and appears to be more challenging.

		\begin{figure}[ht]
			\centering    \includegraphics[angle=-90,width=0.89\textwidth]{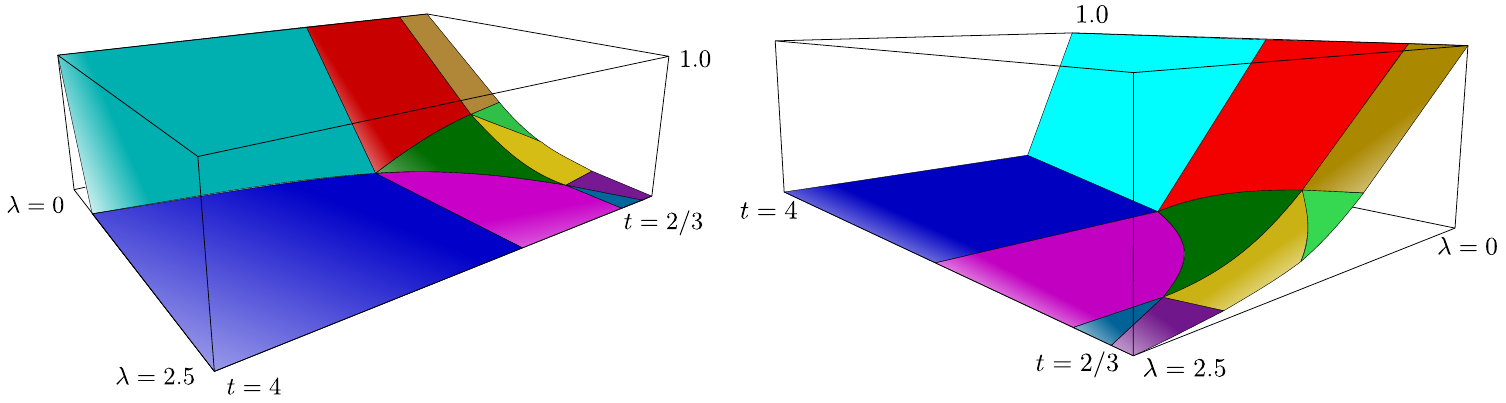}
			\caption{A $3$d representation of $G_t(\lambda)$, for $0\le \lambda\le 5/2$ and $2/3\le t\le 4$, seen from two different viewpoints.
				(The interested reader can further interact~\cite{Interact2023} with the graph of $G_t(\lambda)$ with the observer located in different positions.)
			}
			\label{FigureG3d}
		\end{figure}

		\section{Main results}
		We begin by calculating the explicit formula of $G_t(\lambda)$ in two cases. 
		First, when $t$ is large and the lattice points are seen in the natural order 
		when sweeping the horizon with rays sent at increasing angles towards the lattice points. 
		And then when $t$ is slightly smaller, and the first changes in the natural order 
		in which the target points are seen occur.
		\begin{theorem}\label{T:GLambdaClosedForm}
			Let $G_t(\lambda)$ be the gap distribution function defined above.
			\begin{enumerate}
				\item When $t>2$, we have
				\begin{equation}\label{eqG1infty}
					G_t(\lambda) = \begin{cases}
						1-\ldfrac{\lambda t}{2}, & \text{if } 0\leq\lambda\leq\ldfrac{2}{t}; \\[3pt]
						0, &\text{if } \ldfrac{2}{t}\leq \lambda.
					\end{cases}
				\end{equation}
				\item When $1<t\leq2$, we have
				\begin{equation}\label{eqG12}
					G_t(\lambda) = \begin{cases}
						1+\ldfrac{\lambda t}{2} + 
						\lambda t\log\big(\ldfrac{t}{2}\big) - \ldfrac{\lambda t^2}{2}, 
						&\text{if } 0\leq\lambda\leq 1; \\[4pt]
						1-\ldfrac{\lambda t}{2} + \lambda t\log\big(\ldfrac{\lambda t}{2}\big) - \ldfrac{\lambda t^2}{2}+t, 
						&\text{if } 1\leq\lambda\leq\ldfrac{2}{t}; \\[4pt]
						0, &\text{if } \ldfrac{2}{t}\leq \lambda.
					\end{cases}
				\end{equation}
			\end{enumerate}
		\end{theorem}
		
		\begin{figure}[ht]
			\centering    
			\includegraphics[angle=0,width=0.79\textwidth]{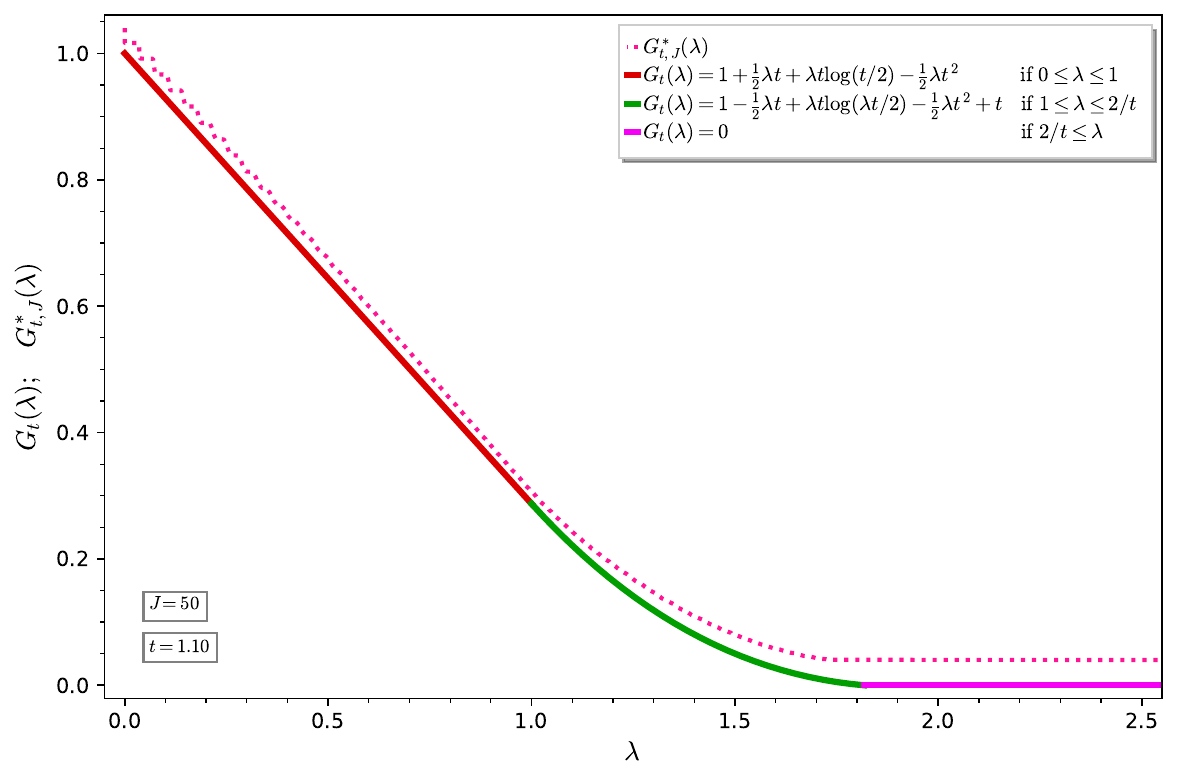}
			\vspace*{-10pt}
			\caption{The graph of the limit gap distribution function $G_t(\lambda)$ 
				compared with the approximation $G_{t,J}^*(\lambda)$ obtained from the
				point of view $(-tJ^2,0)$, with $t=1.1$ and $J=50$.
				The latter drawn with a dotted line is the same as $G_{t,J}(\lambda)$ slightly translated upwards 
				to make the small distinction visible.
			}
			\label{FigureGG1}
		\end{figure}
		The graph of $G_t(\lambda)$ is shown in Figure~\ref{FigureGG1} for $t=1.1$. 
		For a closer look at the target lattice points, with $t=0.82$, the graph is shown below in Figure~\ref{FigureGG2} 
		after a general discussion on the intervals where $G_t(\lambda)$ has different expressions. 
		In both graphs, one can notice that there is a nearly perfect match between the limit $G_t(\lambda)$ 
		and the partial gap distribution functions obtained for even relatively small-sized target squares. 
		In the next two theorems, we prove the existence of $G_t(\lambda)$ in all cases and 
		determine the general shape of the surface $(t, \lambda) \mapsto G_t(\lambda)$.
		A $3$D graphical representation of the generated surface can be seen in Figure~\ref{FigureG3d}.
		
		\begin{theorem}\label{T:GLambdaAllt}
			Let
			\begin{gather*}
				\begin{split}
					R_{k,r,w} &:= \max\Big\{\sdfrac{kt}{w^2}-\sdfrac{1}{w}, \sdfrac{1}{w}-\sdfrac{(r+1)t}{w^2}\Big\}, \\
					S_{k,r,w} &:= \min\Big\{\sdfrac{(k+1)t}{w^2}-\sdfrac{1}{w}, \sdfrac{1}{w}-\sdfrac{rt}{w^2}\Big\}, 
				\end{split} \\
				y_w := \frac{\lambda t}{2w}, \\
				A_{k,r,w}(x) := \Big\{z\in[R_{k,r,w}+x, S_{k,r,w}+x] \colon \min_{\substack{-k\leq j\leq r \\ j\neq 0}}\{jz\}\geq y_w\Big\}, \\
				\begin{split}
					M_{k,r}^- &:= \max\Big\{\sdfrac{\lambda t}{2},\ \sdfrac{(k+r)t}{2}\Big\}, \\
					M_{k,r}^+ &:= \min\Big\{1,\ \sdfrac{(k+r+2)t}{2}\Big\}, 
				\end{split} \\
				h := \Big\lfloor\frac{2}{t}\Big\rfloor.
			\end{gather*}
			Then, for all $t>0$,
			\begin{equation*}
				G_{t}(\lambda) = \frac{1}{2}\sum_{\substack{0\leq k,r\leq h \\ M_{k,r}^- \leq M_{k,r}^+}} \left(\int_{M_{k,r}^-}^{M_{k,r}^+} \int_0^1 w\mu\left(A_{k,r,w}(x)\right)\,\dv x\dv w\right),
			\end{equation*}
			where $\mu$ is the Lebesgue measure.
		\end{theorem}
		
		\begin{theorem}\label{T:GLambdaAlltConstants}
			There is a decreasing sequence of numbers $2=T_0 > T_1 > T_2 > \cdots$ converging to $0$, satisfying the following two properties:
			\begin{enumerate}
				\item For any positive integer $h$, $T_j\in[2/(h+1), 2/h]$ implies $T_j$ is a rational number with denominator at most $2h(h+1)$.
				\item For any interval $[T_j, T_{j+1}]$, there is a finite sequence 
				\begin{equation*}
					0 = \Lambda_{j,0}(t) \leq \Lambda_{j,1}(t) \leq \cdots 
					\leq \Lambda_{j,M_j}(t) \leq \Lambda_{j, M_j+1}(t) = 2/t,
				\end{equation*}
				where every inequality is valid for all $t\in[T_j, T_{j+1}]$.
			\end{enumerate}
			For any pair of intervals $[T_{j}, T_{j+1}]$ and $[\Lambda_{j,i}(t), \Lambda_{j,i+1}(t)]$, there exist constants $\kappa_{j,i,k}$ ($k\in\{1,2,\dots,8\}$) such that, for any $t\in[T_{j}, T_{j+1}]$ and $\lambda\in[\Lambda_{j,i}(t), \Lambda_{j,i+1}(t)]$,
			\begin{align*}
				G_t(\lambda) &= \kappa_{j,i,1} + \kappa_{j,i,2}t + \kappa_{j,i,3}\lambda t + \kappa_{j,i,4}\lambda t^2 \\
				&\phantom{=}+ \kappa_{j,i,5}t\log(\lambda) + \kappa_{j,i,6}\lambda t\log(\lambda) \\
				&\phantom{=}+ \kappa_{j,i,7}t\log(t) + \kappa_{j,i,8}\lambda t\log(t).
			\end{align*}
		\end{theorem}
		
		Theorem~\ref{T:GLambdaAlltConstants} can be reformulated by using the partitions for $t$ and $\lambda$ to describe two-dimensional regions in the $(\lambda, t)$-plane. Furthermore, one can find relations between the constants in bordering regions.
		\begin{corollary}\label{C:GLambdaAlltConstants}
			We can split the first quadrant of the $(\lambda, t)$-plane into a countable number of regions $\cD_0, \cD_1, \cD_2, \dots$, such that in each region $\cD_i$ there are constants $\kappa_{i, k}$ ($k\in\{1,2,\dots,8\}$) such that, for any $(\lambda, t)\in\cD_i$,
			\begin{equation*}
				\begin{split}
					G_t(\lambda) 
					=& \kappa_{i,1} + \kappa_{i,2}t + \kappa_{i,3}\lambda t + \kappa_{i,4}\lambda t^2 \\
					& + \kappa_{i,5}t\log(\lambda) + \kappa_{i,6}\lambda t\log(\lambda) \\
					& + \kappa_{i,7}t\log(t) + \kappa_{i,8}\lambda t\log(t).
				\end{split}
			\end{equation*}
			Each region $\cD_i$ is bounded by lines of the form $\lambda=c$, $t=c$, or $\lambda t = c$ where the $c$ are rational numbers. Furthermore, we have the following relations between constants in neighboring regions. If $\cD_i$ and $\cD_j$ share a border along a line $t=c$, then
				\begin{equation} \label{E:FirstBoundaryRelation}  
					\begin{split}
						\kappa_{i,1} + \kappa_{i,2}c + \kappa_{i,7}c\log(c) 
						&= \kappa_{j,1} + \kappa_{j,2}c + \kappa_{j,7}c\log(c), \\
						\kappa_{i, 3} + \kappa_{i, 4}c + \kappa_{i, 8}\log(c) 
						&= \kappa_{j, 3} + \kappa_{j, 4}c + \kappa_{j, 8}\log(c), \\
						\kappa_{i,5} &= \kappa_{j,5}, \\
						\kappa_{i, 6} &= \kappa_{j,6}.
					\end{split}        
				\end{equation}
			If $\cD_i$ and $\cD_j$ share a border along a line $\lambda=c$, then
				\begin{equation} \label{E:MiddleBoundaryRelation}  
					\begin{split}
						\kappa_{i,1}& = \kappa_{j,1}, \\
						\kappa_{i,2} + \kappa_{i,3}c + \kappa_{i,5}\log(c) + \kappa_{i,6}c\log(c) 
						&= \kappa_{j,2} + \kappa_{j,3}c + \kappa_{j,5}\log(c) + \kappa_{j,6}c\log(c), \\
						\kappa_{i,4} &= \kappa_{j,4}, \\
						\kappa_{i,7} + \kappa_{i,8}c &= \kappa_{j,7} + \kappa_{j,8}c.
					\end{split}        
				\end{equation}
			If $\cD_i$ and $\cD_j$ share a border along a line $\lambda t=c$, then
			\begin{equation} \label{E:LastBoundaryRelation}
				\begin{split}
					\kappa_{i,1} + \kappa_{i,3}c + \kappa_{i,6}c\log(c) 
					&= \kappa_{j,1} + \kappa_{j,3}c + \kappa_{j,6}c\log(c), \\
					\kappa_{i,2} + \kappa_{i,4}c + \kappa_{i,5}\log(c) 
					&= \kappa_{j,2} + \kappa_{j,4}c + \kappa_{j,5}\log(c), \\
					\kappa_{i,7} - \kappa_{i,5} 
					&= \kappa_{j,7} - \kappa_{j,5}, \\
					\kappa_{i,8} - \kappa_{i,6} 
					&= \kappa_{j,8} - \kappa_{j,6}.
				\end{split}        
			\end{equation}
		\end{corollary}
		
		\begin{table}[ht]
			\setlength{\tabcolsep}{5pt}
			\renewcommand{\arraystretch}{1.15} 
			\centering
			\caption{Definition of constants $\kappa_{i,j}$ for $0\le i\le 7$ and $1\le j\le 8$.}
			\begin{tabular}{c@{\hskip 2.5em}cccccccc}
				\toprule
				$i$ & $\kappa_{i,1}$ & $\kappa_{i,2}$ & $\kappa_{i,3}$ & $\kappa_{i,4}$ & $\kappa_{i,5}$ & $\kappa_{i,6}$ & $\kappa_{i,7}$ & $\kappa_{i,8}$ \\
				\midrule
				$0$ & $0$ & $0$ & $0$ & $0$ & $0$ & $0$ & $0$ & $0$ \\
				$1$ & $1$ & $0$ & $-1/2$ & $0$ & $0$ & $0$ & $0$ & $0$ \\
				$2$ & $1$ & $0$ & 1/2 - $\log(2)$ & -1/2 & 0 & 0 & 0 & 1 \\
				$3$ & $1$ & $1$ & -1/2 - $\log(2)$ & -1/2 & 0 & 1 & 0 & 1 \\
				$4$ & $1$ & $0$ & 1 - $\log(2)$ & -1 & 0 & 0 & 0 & 3/2 \\
				$5$ & $1$ & $5$ & -4 - $\log(2)$ & -1 & 2 & 3 & 0 & 3/2 \\
				$6$ & $-1$ & $3$ & -2 - $\log(2)$ & 1 & 0 & 1 & -2 & -1/2 \\
				$7$ & $-1$ & $-2$ + $2\log(2)$ & $1/2$ + $\log(2)/2$ & $1$ & $-2$ & $-1/2$ & $-2$ & $-1/2$\\
				\bottomrule
			\end{tabular}
			\label{TableKi}
		\end{table}
		
		One can use Theorem~\ref{T:GLambdaAllt} to calculate explicit formulas for $G_t(\lambda)$ in any case $2/(h+1) < t \leq 2/h$ where $h$ is a non-negative integer, with cases for higher values of $h$ taking longer to calculate than those for lower values of $h$. (Those readers interested in calculating further cases will find the proof of Lemma~\ref{L:AlltOuterSum} useful, in particular \eqref{E:IntegralEqualsIntervalMeasures}.) Here we give the next case, that is when $h=2$.
		\begin{corollary}\label{C:GLambda23t1}
			When $2/3<t\leq1$,
			\begin{equation}\label{eqG2pe31}
				G_t(\lambda) =
				\begin{cases}
					1+\lambda t - \lambda t^2 + \ldfrac{3}{2}\lambda t\log(t) - \lambda t\log(2), & \text{if } 0\leq \lambda \leq 1; \\[4pt]
					1 + 5t - 4\lambda t - \lambda t^2 + t(2+3\lambda)\log(\lambda) + \ldfrac{3}{2}\lambda t \log(t) - \lambda t \log(2), & \text{if } 1\leq\lambda\leq\ldfrac{1}{t}; \\[4pt]
					-1 + 3t - 2\lambda t + \lambda t^2 + \lambda t\log\big(\ldfrac{\lambda}{2}\big) - t\left(2+\ldfrac{1}{2}\lambda\right) \log(t), &\text{if } \ldfrac{1}{t}\leq\lambda\leq 2; \\[4pt]
					-1 - 2t + \ldfrac{1}{2}\lambda t + \lambda t^2 - t\left(2+\ldfrac{1}{2}\lambda\right)\log\big(\ldfrac{\lambda t}{2}\big), &\text{if } 2\leq\lambda\leq\ldfrac{2}{t}; \\[4pt]
					0, &\text{if } \ldfrac{2}{t} \leq \lambda.
				\end{cases}
			\end{equation}
		\end{corollary}
		
		\begin{figure}[ht]
			\centering    
			\includegraphics[angle=0,width=0.79\textwidth]{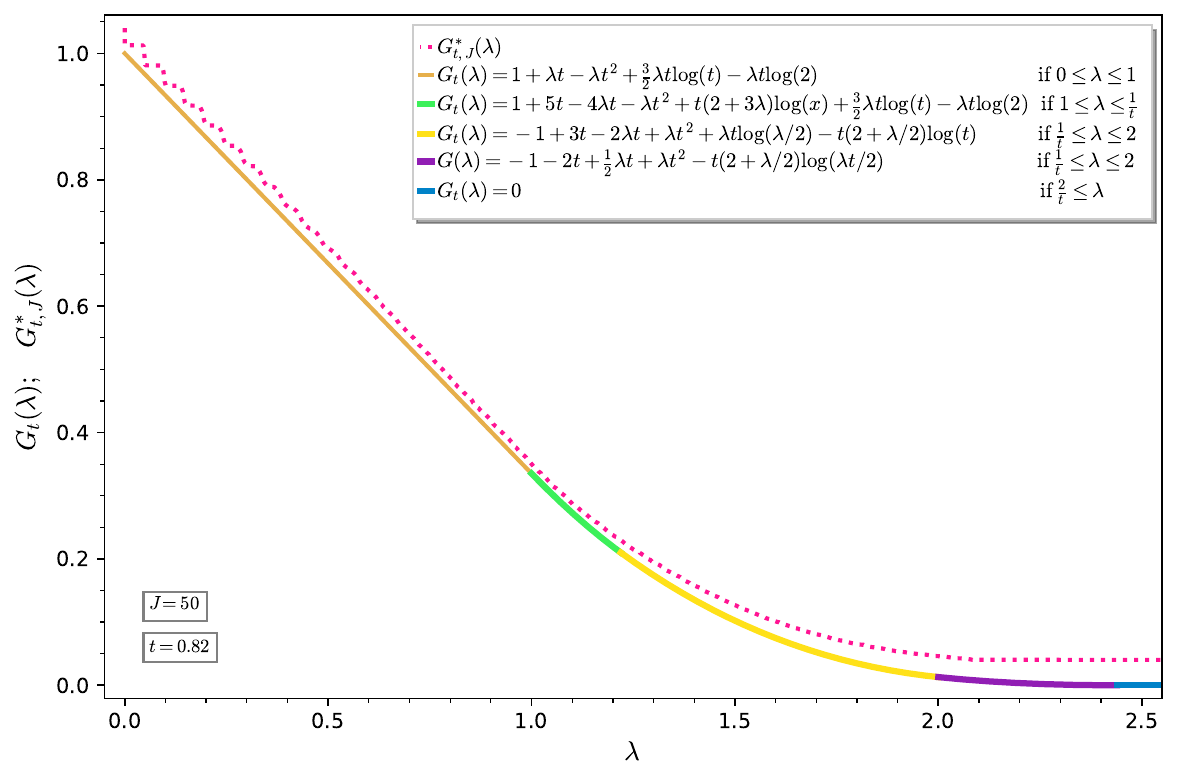}
			\vspace*{-10pt}
			\caption{The graph of the limit gap distribution function $G_t(\lambda)$ 
				compared with the partial one $G_{t,J}^*(\lambda)$ obtained from the
				point of view $(-tJ^2,0)$, with $t=0.82$ and $J=50$.
				The latter drawn with a dotted line coincides with $G_{t,J}(\lambda)$ except that it 
				is slightly translated upwards to make the distinction visible.
			}
			\label{FigureGG2}
		\end{figure}
		
		From Theorem~\ref{T:GLambdaClosedForm} and Corollary~\ref{C:GLambda23t1}, we can deduce the first few regions $\cD_i$ described in Corollary~\ref{C:GLambdaAlltConstants}.
		Their graphic representation is shown in Figure~\ref{FigureDomainsD0-7} and their precise definitions are: 
		\begin{align}\label{eqDomainsD0-7}
			\begin{aligned}
				\cD_0 &= \big\{(\lambda, t) : 2 \leq \lambda t \big\},\\
				\cD_2 &= \big\{(\lambda, t) : 1 \leq t \leq 2,\ 0 \leq \lambda \leq 1\big\}, \\
				\cD_4 &= \big\{(\lambda, t) : 2/3\leq t\leq 1,\ 0\leq \lambda\leq 1\big\},\\
				\cD_6 &= \big\{(\lambda, t) : 2/3 \leq t\leq 1,\ 1\leq \lambda t,\ \lambda \leq 2\big\},
			\end{aligned}
			&&
			\begin{aligned}
				\cD_1 &= \big\{(\lambda, t) : 2\leq t,\ 0 \leq \lambda t \leq 2\big\},\\
				\cD_3 &= \big\{(\lambda, t) : 1 \leq t \leq 2,\ 1\leq \lambda,\ \lambda t \leq 2\big\}, \\
				\cD_5 &= \big\{(\lambda, t) : 2/3\leq t\leq 1,\ 1\leq \lambda,\ \lambda t\leq 1\big\},\\
				\cD_7 &= \big\{(\lambda, t) : 2/3 \leq t\leq 1,\  2\leq \lambda,\ \lambda t \leq 2\big\}.
			\end{aligned}
		\end{align}
		
		Let us group the constants $\kappa_{i,1}$ into tuples of eight each, 
		which we denote by
		$K_i = (\kappa_{i,1}, \kappa_{i,2},\dots, \kappa_{i,8})$  
		for $0\le i\le 7$. 
		The explicit values of these constants are given in Table~\ref{TableKi}.
		One can directly check that these constants do indeed satisfy relations~\eqref{E:FirstBoundaryRelation},~\eqref{E:MiddleBoundaryRelation} and~\eqref{E:LastBoundaryRelation}.
		\begin{figure}[ht]
			\centering    
			\includegraphics[angle=0,width=0.64\textwidth]{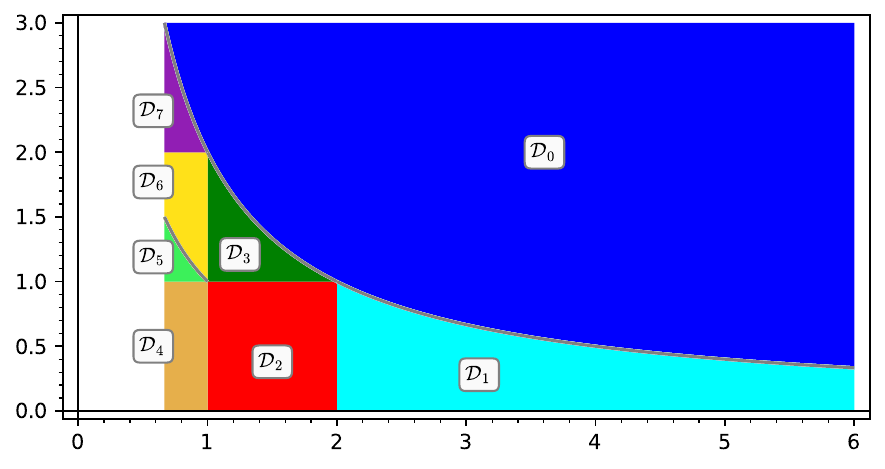}
			\vspace*{-10pt}
			\caption{The domains $\cD_0,\cD_1,\dots,\cD_7$ defined by~\eqref{eqDomainsD0-7}. 
			} \label{FigureDomainsD0-7}
		\end{figure}

		Formulas~\eqref{eqG1infty}, \eqref{eqG12} and~\eqref{eqG2pe31} of the gap distribution function allow us to obtain the corresponding density function for each of the intervals $(2,\infty)$, $(1,2]$ and $(2/3,1]$.
		We will discuss this aspect and provide explicit formulas and graphical representations of the density in Section~\ref{SectionDensities}.

		\subsection{The density of the gap distribution}\label{SectionDensities}
		Having the explicit formulas for the gap distribution function $G_t(\lambda)$ 
		for $t$ in the intervals $(2,\infty)$, $(1,2]$ and $(2/3,1]$, 
		we can now deduce the corresponding density function $g_t(\lambda)$ on these intervals.
		We define $g_t(\lambda)$  as the limit
		\begin{equation*}
			g_t(\lambda) = \lim_{\delta \searrow 0}\frac{1}{2\delta}\lim_{J\to\infty}
			\frac{1}{2J^2}\# \Big\{ \frac{\angle\big((a,m)P_{t,J}\widetilde{(a,m)}\big)}{\Delta_{av}}\in [\lambda-\delta,\lambda+\delta) \Big\}\,.
		\end{equation*}
		One sees that the inner quantity above closely resembles the definition of $G_{t,J}(\lambda)$. Thus we have,
		\begin{equation*}
			g_t(\lambda) 
			= \lim_{\delta \searrow 0}\frac{1}{2\delta} 
			\lim_{J\to\infty}
			\big(G_{t,J}(t-\lambda)-G_{t,J}(t+\lambda)\big)
			= \lim_{\delta \searrow 0}\frac{1}{2\delta} \big( G_t(\lambda-\delta) - G_t(\lambda+\delta) \big).
		\end{equation*}
		Further, this equals
		\begin{equation*}
			g_t(\lambda) 
			= \frac 12 \lim_{\delta \searrow 0}
			\left(\frac{G_t(\lambda-\delta)-G_{t}(\lambda)}{\delta} 
			- \frac{G_t(\lambda+\delta)-G_{t}(\lambda)}{\delta}
			\right)
			= - \frac 12\left(
			\frac{\partial}{\partial\lambda}G_t(\lambda^{-})
			+ \frac{\partial}{\partial\lambda}G_t(\lambda^{+})
			\right).
		\end{equation*}
		Hence, except at the boundary of the domains on which the limit exists, we have $g_t(\lambda)=-\frac{\partial}{\partial \lambda}G_t(\lambda)$.
		At the endpoints, we see that the values of the density will be the average of the left 
		and right derivatives at those points.
		Also, for any $t>0$ and any $\lambda\geq 0$, we have
		\begin{equation*}
			G_t(\lambda)=\int_\lambda^\infty g_t(u)\,\dv u\,.
		\end{equation*}
		
		A 3D representation of the limit density function $g_t(\lambda)$ can be seen in Figure~\ref{Figureg3d}.
		Note the rupture on the surface along
		the curve $\lambda=2/t$ for $t>1$ 
		and the range where the density vanishes.
		
		\begin{figure}[htb]
			\centering    \includegraphics[angle=-90,width=0.89\textwidth]{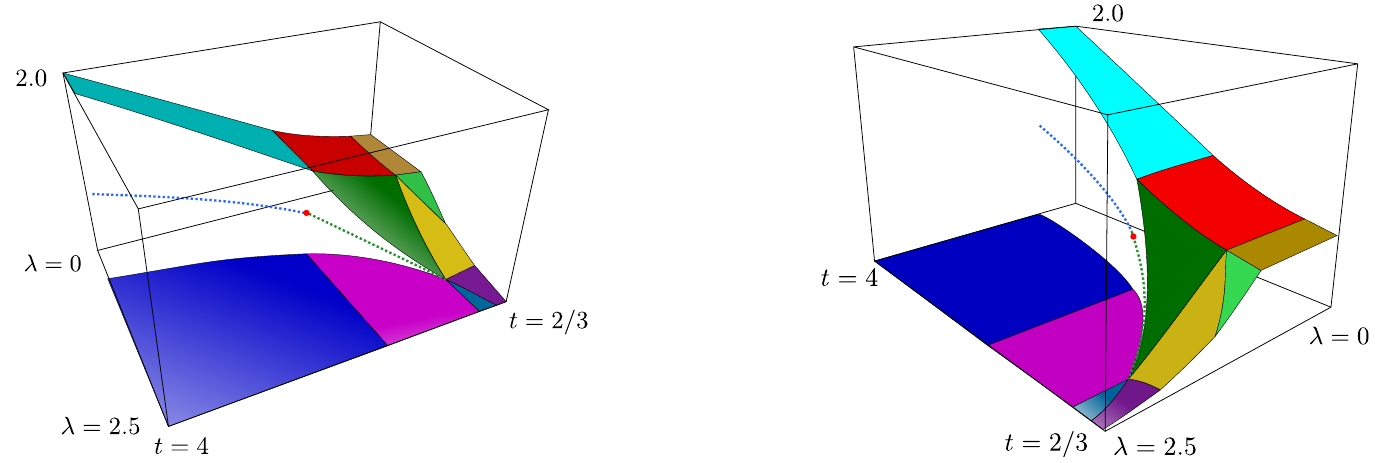}
			\caption{A $3$d representation of $g_t(\lambda)$, for $0\le \lambda\le 5/2$ and $2/3\le t\le 4$, seen from two different viewpoints.
			}
			\label{Figureg3d}
		\end{figure}
		\bigskip
		
		The exact expressions of the $g_t(\lambda)$ for $t > 2/3$ are as follows.
		1. If $2<t$,
		\begin{equation*}
			g_t(\lambda) = 
			\begin{cases}
				\ldfrac t2 & \text{ if $0<\lambda< \ldfrac 2t$,}\\[4pt]
				\ldfrac t4 & \text{ if $\lambda = \ldfrac 2t$,}\\[4pt]
				0 & \text{ if $ \ldfrac 2t <\lambda$.}
			\end{cases}
		\end{equation*}
		
		2. If $1<t\le 2$,
		\begin{equation*}
			g_t(\lambda) = 
			\begin{cases}
				\frac{t^2}{2} -\frac{t}2 -t\log\big(\frac{t}{2}\big) & \text{ if $0<\lambda< 1$,}\\[4pt]
				\frac{t^2}{2} -\frac{t}{2} -t\log\big(\frac{\lambda t}{2}\big) & \text{ if $1\le\lambda< \frac{2}{t}$,}\\[4pt]    
				\frac{t^2}{4} - \frac{t}{4} & \text{ if $\lambda = \frac{2}{t}$,}\\[4pt]
				0 & \text{ if $ \frac 2t <\lambda$.}
			\end{cases}
		\end{equation*}
		
		3. If $2/3<t\le 1$,
		\begin{equation*}
			g_t(\lambda) = 
			\begin{cases}
				t^2 -t -\frac{3}{2}t\log(t) +t\log(2)& \text{ if $0<\lambda< 1$,}\\[8pt]
				t^2 +t -\frac{2t}{\lambda} -3t\log(\lambda) -\frac{3t}{2}\log(t) +t\log(2) & \text{ if $1\le\lambda< \sdfrac 1t$,}\\[4pt]    
				-t^2 +t -t\log\big(\frac{\lambda}{2}\big) +\frac{t}{2}\log(t) & \text{ if $\frac 1t\le \lambda \le 2$,}\\[4pt]
				-t^2 +\frac{2t}{\lambda} +\frac{t}{2}\log\big(\frac{\lambda t}{2}\big) & \text{ if $2\le \lambda \le \frac 2t$,}\\[4pt]
				0 & \text{ if $ \frac 2t <\lambda$.}
			\end{cases}
		\end{equation*}
		
		Furthermore, let $\cD_0$, $\cD_1$, $\cD_2,\dots$, be the regions described in Corollary~\ref{C:GLambdaAlltConstants}. Then, for any region $\Int(\cD_j)$ (that is, the interior of $\cD_j$), we have constants $\alpha_{j,1}$, $\alpha_{j,2}$, $\alpha_{j,3}$, $\alpha_{j,4}$, $\alpha_{j,5}$ such that, for any $\lambda, t \in \Int(\cD_j)$,
		\begin{equation*}
			\begin{split}
				g_t(\lambda) &= \alpha_{j,1}t + \alpha_{j,2}t^2 
				+ \alpha_{j,3}\sdfrac{t}{\lambda} 
				+ \alpha_{j,4} t\log(\lambda) + \alpha_{j,5} t\log(t).
			\end{split}
		\end{equation*}
		In particular, we have
		\begin{align*}
			\hspace*{26mm}
			&&&&& 
			\hspace*{26mm}
			\\[-6mm]
			\alpha_{j,1} &= - \kappa_{j,3} - \kappa_{j,6}, &
			\alpha_{j,2} &= - \kappa_{j,4}, &
			\alpha_{j,3} &= - \kappa_{j, 5}, \\
			\alpha_{j,4} &= - \kappa_{j,6}, &
			\alpha_{j,5} &= - \kappa_{j,8}.
		\end{align*}
		
		\medskip
		
		From now on in this paper, without further specifications, 
		we will assume that all implied constants depend on 
		$\lambda$ and $t$, and so we will drop such subscripts,
		writing $O$ to implicitly mean $O_{\lambda, t}$, and, 
		likewise,~$\ll$ to mean $\ll_{\lambda, t}$. 
		Also, for two points $Q_1$ and $Q_2$, we may refer 
		to the angle $\angle (Q_1 P_{t,J} Q_2)$ 
		simply as the angle between $Q_1$ and~$Q_2$.
		
		The remainder of our manuscript is organized as follows: 
		In Section~\ref{S:tBiggerThan2} we will find a formula for $G_t(\lambda)$ in the simplest case, when $t>2$. In Section~\ref{S:tBetween1And2} we move onto the next case, that is, $1<t\leq 2$. This will introduce some ideas which lay out the blueprint for proving the existence of $G_t(\lambda)$ for all $t>0$, in particular, that of a notion which we call interference. We say that `a point $Q$ has interference' (or less specifically `we have interference') from a particular line if the point seen after it could potentially lie on that line. For example, in the simplest case $t>2$ we have no interference, whereas in the intermediate case $1<t\leq 2$ we begin to see interference from at most one line away; it is this interference which adds complexity to dealing with further cases beyond the first. Section~\ref{S:AllValuesOft} is dedicated mostly to proving the aforementioned existence of $G_t(\lambda)$ as given in Theorem~\ref{T:GLambdaAllt}. We then finish off the section by justifying the formula of $G_t(\lambda)$ described in Theorem~\ref{T:GLambdaAlltConstants}.
		
		Additional representations of the gap distribution function $G_t(\lambda)$ and its density $g_t(\lambda)$, with the additional possibility to interactively choose the observer's position or change the involved parameters, can be found by following the hyperlink in the reference~\cite{Interact2023}. 
		
		\section{\texorpdfstring{An explicit formula for $t>2$}{An explicit formula for t greater than 2}}\label{S:tBiggerThan2}
		
		For any $m\in\{0,\dots,J-1\}$, consider the line passing through the points $(-J,m)$ and $(J,m+1)$. This line will intersect the $x$-axis at
		\begin{equation*}
			x = -J(2m+1) \geq -2J^2-J.
		\end{equation*}
		So, for any fixed $t>2$, there is $J$ large enough such that $-tJ^2<-2J^2-J$, and thus the angle between $(-J,m)$ and the $x$-axis is smaller than the angle between $(J,m+1)$ and the $x$-axis. In other words, for any point $Q$ which does not lie on the left edge of the rectangle, the next point seen after it by the observer (which we shall denote by $\widehat{Q}$) will lie on the same line.
		
		For any two points $Q_1=(a_1, m)$ and $Q_2=(a_2, m)$ on the same line $y=m$, let $D=\abs{a_2-a_1}$ denote the distance between them. Let $\alpha$ be the angle between $Q_1$ and $Q_2$. Then, by equating two different formulas for the area of the triangle between $P_{t,J}$, $Q_1$, and $Q_2$, we find
		\begin{equation*}
			\sin(\alpha)\sqrt{(tJ^2+a_1)^2+m^2} \sqrt{(tJ^2+a_2)^2+m^2} = Dm.
		\end{equation*}
		Since, for any $-J\leq a\leq J$,
		\begin{equation*}
			\sqrt{(tJ^2+a)^2+m^2} 
			= tJ^2\left(1+O\Big(\sdfrac{1}{J}\Big)\right),
		\end{equation*}
		we find
		\begin{equation}
			\sin(\alpha) = \sdfrac{Dm}{t^2J^4}
			\left(1+O\Big(\sdfrac{1}{J}\Big)\right).
		\end{equation}
		To calculate $G_{t,J}(\lambda)$, we note that for all but one point on the line $y=m$, the angle between that point and the next point is
		\begin{equation*}
			\alpha = \frac{m}{t^2J^4}\left(1+O\Big(\sdfrac{1}{J}\Big)\right).
		\end{equation*}
		Therefore, using $\Delta_{av}$ as found in \eqref{E:AverageAngle},
		\begin{align*}
			G_{t,J}(\lambda) &= \frac{1}{2J^2+O(J)}\left[
			\#\Big\{(a,m)\in\cR_J 
			:
			\frac{m}{t^2J^4} +O\left(\sdfrac{1}{J^4}\right) 
			\geq \sdfrac{\lambda}{2tJ^3}+O\left(\sdfrac{1}{J^4}\right)\Big\} + O(J)\right] \\
			&=\frac{1}{2J^2+O(J)}\left[
			\#\Big\{(a,m)\in\cR_J 
			: 
			m\geq \sdfrac{\lambda tJ}{2} + O(1)\Big\} + O(J)\right] \\
			&=\frac{1}{2J^2+O(J)}\left[\#\left\{(a,m)\in\cR_J 
			:
			m\geq \frac{\lambda tJ}{2}\right\} + O(J)\right] \\
			&=\sdfrac{1}{J}\left(1+O\left(\fdfrac{1}{J}\right)\right)
			\#\Big\{m\in\Z 
			:
			\sdfrac{\lambda tJ}{2}\leq m \leq J\Big\} + O\left(\fdfrac{1}{J}\right).
		\end{align*}
		This essentially reduces to a formula in two cases, so we obtain
		\begin{equation*}
			\begin{split}
				G_{t,J}(\lambda) 
				&= \fdfrac{1}{J}
				\left(1+O\left(\ldfrac{1}{J}\right)\right)
				\begin{cases}
					J-\frac{\lambda tJ}{2} + O(1), &\text{if } \frac{\lambda tJ}{2} \leq J \\[4pt]
					0, &\text{if } \frac{\lambda tJ}{2} > J
				\end{cases} 
				\ + O\left(\fdfrac{1}{J}\right) \\
				&=\begin{cases}
					1-\frac{\lambda t}{2} + O\left(\frac{1}{J}\right), &\text{if } \lambda \leq \frac{2}{t}; \\[4pt]
					O\left(\frac{1}{J}\right), &\text{if } \lambda > \frac{2}{t}.
				\end{cases}
			\end{split}
		\end{equation*}
		Finally, by letting the side length $J$ of the box go to infinity, 
		we obtain the limit formula for $G_t(\lambda)$ when~$t> 2$:
		\begin{equation*}
			G_{t}(\lambda) := \lim_{J\to\infty} G_{t,J}(\lambda) 
			= \begin{cases}
				1-\frac{\lambda t}{2}, &\text{if } \lambda \leq \frac{2}{t}; 
				\\[4pt]
				0, &\text{if } \lambda > \frac{2}{t}.
			\end{cases}
		\end{equation*}
		
		\section{\texorpdfstring{An explicit formula for $1<t\leq2$}{An explicit formula for 1 less than t less than 2}}\label{S:tBetween1And2}
		
		We now turn our attention to the case in Theorem~\ref{T:GLambdaClosedForm} when $1<t\leq2$. The line connecting $(-J,m)$ and $(J,m+2)$ will intersect the $x$-axis at
		\begin{equation*}
			x = -mJ-J \geq -J^2-J.
		\end{equation*}
		So, for any fixed $t$ between $1$ and $2$, for $J$ sufficiently large, all points on the line $y=m$ will be seen before any point on the line $y=m+2$. So, for any point $Q$, the next point that is seen $\widehat{Q}$ will be at most one line away (in other words $\widehat{Q}$ will lie on the line $y=m-1$, $y=m$, or $y=m+1$). Let $A_m$ be the $x$-coordinate of the intersection between the line joining $P_{t,J}$ to $(J,m+1)$ and the line $y=m$, and let $B_m$ be the $x$-coordinate of the intersection between the line joining $P_{t,J}$ to $(-J,m-1)$ and the line $y=m$. We find,
		\begin{equation*}
			\begin{split}
				A_m &= J-\sdfrac{tJ^2}{m} + O(1), \\
				B_m &= \sdfrac{tJ^2}{m} - J + O(1).
			\end{split}
		\end{equation*}
		
		For any point $Q$ on the line $y=m$, $\widehat{Q}$ can lie on the line $y=m+1$ only when $Q$ lies to the left of $(A_m, m)$, and on the line $y=m-1$ only when $Q$ lies to the right of $(B_m, m)$. In order to count the number of points $(a,m)$ for which the angle to the next point is greater than or equal to $\lambda\Delta_{av}$, we will divide our rectangle into three regions, one for which $(a,m)$ lies to the left of $(A_m,m)$, one for which $(a,m)$ is between $(A_m,m)$ and $(B_m, m)$, and one for which $(a,m)$ lies to the right of $(B_m, m)$. We will denote by $\Sigma_1$ the total number of points (satisfying our angle condition) in the first region, by $\Sigma_0$ those points in the second region, and by $\Sigma_{-1}$ those points in the third region. Then,
		\begin{equation}
			G_{t,J}(\lambda) = \frac{1}{2J^2+O(J)}\left(\Sigma_1+\Sigma_0+\Sigma_{-1}\right).
		\end{equation}
		In particular, $\Sigma_1$, $\Sigma_0$, and $\Sigma_{-1}$ split the rectangle into regions where we have interference from the line above, no interference, and interference from the line below respectively.
		
		\subsection{\texorpdfstring{Calculating $\Sigma_1$ -- Interference from the Line Above}{Calculating Sigma1 -- Interference from the Line Above}}\label{SS:Sigma1}
		
		To find the angle between $Q=(a,m)$ and $\widehat{Q}$, we consider the projection $\tilde{Q}=(s,m+1)$ of $Q$ onto the line $y=m+1$. The distance between~$\tilde{Q}$ and the next integer point to its left will be
		\begin{equation*}
			\{s\} = \Big\{a+\sdfrac{tJ^2+a}{m}\Big\} 
			= \Big\{\sdfrac{tJ^2+a}{m}\Big\}.
		\end{equation*}
		Since $\widehat{Q}$ will be either this integer point or $(a-1,m)$, we find that the angle between $Q$ and $\widehat{Q}$ will be
		\begin{equation}\label{E:SimplifyMinsAngles}
			\begin{split}
				\min\left\{
				\Big\{\sdfrac{tJ^2+a}{m}\Big\}\sdfrac{m+1}{t^2J^4} 
				+ O\left(\sdfrac{1}{J^4}\right),\ 
				\sdfrac{m}{t^2J^4}
				+O\left(\sdfrac{1}{J^4}\right)
				\right\}  
				= \Big\{\sdfrac{tJ^2+a}{m}\Big\}\sdfrac{m}{t^2J^4} 
				+ O\left(\sdfrac{1}{J^4}\right), 
			\end{split}
		\end{equation}
		where the equality arises from noticing that 
		\begin{equation*}
			\sdfrac{m}{t^2J^4}+O\left(\sdfrac{1}{J^4}\right) \leq \{s\}\sdfrac{m+1}{t^2J^4}+O\left(\sdfrac{1}{J^4}\right)
		\end{equation*} 
		implies that $1=\{s\}+O(1/J)$. So, we want to count the number of points $(a,m)$ such that $-J\leq a\leq A_m$ and
		\begin{equation*}
			\left\{\sdfrac{tJ^2+a}{m}\right\} 
			\geq \sdfrac{\lambda tJ}{2m}
			\left(1+O\Big(\sdfrac{1}{J}\Big)\right).
		\end{equation*}
		We require that $m\geq tJ/2 + O(1)$, otherwise $A_m < -J$ and $(A_m, m)$ would lie outside the rectangle. We will also require $m\geq \lambda tJ/2 + O(1)$, otherwise all angles on the line $y=m$ will be too small. Writing
		\begin{equation*}
			x_m = \left\{\sdfrac{tJ^2}{m}\right\}
			\text{\quad and\quad }
			y_m = \sdfrac{\lambda tJ}{2m},
		\end{equation*}
		we see that $\Sigma_1$ is
		\begin{equation*}
			\begin{split}
				=& \sum_{\max\{\lambda, 1\}\ldfrac{tJ}{2}+O(1)\leq m\leq J} 
				\hspace*{-7pt}
				\#\left\{\sdfrac{a}{m}\in \sdfrac{\Z}{m}
				\cap\Big[-\sdfrac{J}{m},\, \sdfrac{J}{m}
				-\sdfrac{tJ^2}{m^2}+O\left(\sdfrac{1}{J}\right)\Big] 
				:
				\left\{x_m+\Big\{\sdfrac{a}{m}\Big\}\right\} \geq 
				y_m+O\left(\sdfrac{1}{J}\right)\right\} \\
				=& \sum_{\max\{\lambda, 1\}\ldfrac{tJ}{2}\leq m\leq J} 
				m\mu\bigg(\Big\{u\in 
				\Big[-\frac{J}{m},\,   \frac{J}{m}-\frac{tJ^2}{m^2}
				+O\left(\sdfrac{1}{J}\right)\Big] 
				:
				\left\{x_m+u\right\} \geq y_m
				+O\Big(\sdfrac{1}{J}\Big)\Big\}\bigg) 
				+ O(J) \\
				=& \sum_{\max\{\lambda, 1\}\ldfrac{tJ}{2}\leq m\leq J} 
				m\mu\bigg(\Big\{u\in \left[-\frac{J}{m}, \,\frac{J}{m}-\frac{tJ^2}{m^2}\right] : \left\{x_m+u\right\} \geq y_m\Big\}\bigg) + O(J),
			\end{split}    
		\end{equation*}
		where $\mu$ is the Lebesgue measure. We split the interval $[\max\{\lambda, 1\}tJ/2, J]$ into smaller sub-intervals of length $J^{3/4}$. Let
		\begin{gather*}
			L = J^{1/4}, \\
			C_{\lambda, t} = \max\{1,\lambda\}\frac{t}{2}, \\
			I_i = \left[\sdfrac{iJ}{L},\, \sdfrac{(i+1)J}{L}\right], \\
			R_m = -\frac{J}{m}, \\[1pt]
			S_m = \frac{J}{m} - \frac{tJ^2}{m^2}, \\[5pt]
			A_{R,S,y}(x) = \{u \in [R,S] : \{x+\{u\}\}\geq y\}. \label{E:ARSyxDef}
		\end{gather*}
		Then,
		\begin{equation}
			\Sigma_1 = \sum_{C_{\lambda, t}L\leq i\leq L} \sum_{m\in I_i} m\mu\big(A_{R_m,S_m,y_m}(x_m)\big) + O\left(\sdfrac{J^2}{L}\right). \label{E:Sigma1PreLemmas}
		\end{equation}
		Since the intervals $I_i$ for $C_{\lambda,t}L\leq i\leq L$ don't exactly cover $C_{\lambda, t}J\leq m \leq J$, we under-count or over-count by $O(J/L)$ lines, picking up an error of the size $O(J^2/L)$. To deal with the inner sum, we utilize the following two lemmas, which will be proved below.
		\begin{lemma}\label{L:Sigma1InnerSum}
			Let $A_{R,S,y}(x)$ be as in \eqref{E:ARSyxDef}. Then,
			\begin{equation*}
				\abs{
					\sum_{m\in I_i}\mu\big(A_{R,S,y}(x_m)\big) - \#I_i\int_0^1 \mu(A_{R,S,y}(x))\,\dv x
				} \ll \frac{J^{2/3}}{L} + J^{1/2}.
			\end{equation*}
		\end{lemma}
		\begin{lemma}\label{L:Sigma1Integral}
			Let $A_{R,S,y}(x)$ be as in \eqref{E:ARSyxDef}. Then,
			\begin{equation*}
				\int_0^1 \mu\big(A_{R,S,y}(x)\big)\,\dv x = (S-R)(1-y).
			\end{equation*}
		\end{lemma}
		We first re-index $m$, $R_m$, $S_m$, and $y_m$ in terms of $i$ as
		\begin{gather*}
			m = \frac{iJ}{L}\left(1+O\left(\sdfrac{1}{L}\right)\right), \\
			y_i := \frac{\lambda tL}{2i} = y_m\left(1+O\left(\sdfrac{1}{L}\right)\right), \\
			R_i := -\frac{L}{i} = R_m\left(1+O\left(\sdfrac{1}{L}\right)\right), \\
			S_i := \left(\frac{L}{i}-\frac{tL^2}{i^2}\right) = S_m\left(1+O\left(\sdfrac{1}{L}\right)\right).
		\end{gather*}
		Then, since $\mu\big(A_{R_m, S_m, y_m}(x_m)\big)$ is
		\begin{equation*}
			\begin{split}
				&= \mu\left(\left\{u\in\left[R_i\left(1+O\left(\sdfrac{1}{L}\right)\right), S_i\left(1+O\left(\sdfrac{1}{L}\right)\right)\right] 
				:
				\{x_m + \{u\}\} \geq y_i\left(1+O\left(\sdfrac{1}{L}\right)\right)\right\}\right) \\
				&= \mu\big(\{u\in[R_i, S_i] 
				:
				\{x_m+\{u\}\} \geq y_i\}\big) + O\left(\sdfrac{1}{L}\right) \\
				&= \mu\big(A_{R_i, S_i, y_i}(x_m)\big) + O\left(\sdfrac{1}{L}\right),
			\end{split}  
		\end{equation*}
		
		we can apply Lemmas~\ref{L:Sigma1InnerSum} and \ref{L:Sigma1Integral} to \eqref{E:Sigma1PreLemmas} to find
		\begin{align*}
			\Sigma_1 &= \sum_{C_{\lambda, t}L\leq i \leq L} 
			\sdfrac{iJ}{L}\left(1+O\left(\sdfrac{1}{L}\right)\right) \sum_{m\in I_i} \mu\big(A_{R_i, S_i, y_i}(x_m)\big)
			+ O\left(\sdfrac{J^2}{L}\right) \\
			&= \sum_{C_{\lambda, t}L\leq i\leq L}
			\sdfrac{iJ}{L}
			\left(\sdfrac{J}{L}+O(1)\right)
			\int_0^1 \mu\big(A_{R_i, S_i, y_i}(x)\big)\,\dv x
			+ O\Big(\sdfrac{J^2}{L}+J^{5/3}+J^{3/2}L\Big)\\
			&= \sdfrac{J^2}{L}\sum_{C_{\lambda, t}\leq i\leq L} 
			\sdfrac{i}{L}\left(\sdfrac{2L}{i}-\sdfrac{tL^2}{i^2}\right)
			\left(1-\sdfrac{\lambda tL}{2i}\right)
			+ O\Big(\sdfrac{J^2}{L}+J^{5/3}+J^{3/2}L\Big)\\
			&= \frac{J^2}{L}\sum_{C_{\lambda, t}L\leq i\leq L} 
			\left(2-\frac{tL}{i}(1+\lambda)+\sdfrac{\lambda t^2 L^2}{2i^2}\right) 
			+ O\Big(\sdfrac{J^2}{L}+J^{5/3}+J^{3/2}L\Big).
		\end{align*}
		We now deal with this final sum using the following lemma, which will be proved below.
		\begin{lemma}\label{L:Sigma1OuterSum}
			\begin{equation*}
				\Bigg|
					\sum_{C_{\lambda, t}\leq\frac{i}{L}\leq 1}\left(2-\frac{tL}{i}(1+\lambda)+\frac{\lambda t^2 L^2}{2i^2}\right) 
					- L\int_{C_{\lambda, t}}^1 
					\left(2-\frac{t}{w}(1+\lambda)+\frac{\lambda t^2}{2w^2}\right) \,\dv w
				\Bigg|
				\ll 1.
			\end{equation*}
		\end{lemma}
		Therefore, if $\lambda\leq 2/t$, then
		\begin{align*}
			\Sigma_1 &= J^2\int_{C_{\lambda, t}}^1 
			\left(2-\frac{t}{w}(1+\lambda)+\frac{\lambda t^2}{2w^2}\right) \,\dv w + O\left(\frac{J^2}{L}+J^{5/3}+J^{3/2}L\right) \\
			&= 2J^2(1-C_{\lambda, t}) -tJ^2(1+\lambda)\log\left(\frac{1}{C_{\lambda, t}}\right)-\frac{\lambda t^2 J^2}{2}\left(1-\frac{1}{C_{\lambda, t}}\right) + O\left(\frac{J^2}{L}+J^{5/3}+J^{3/2}L\right).
		\end{align*}
		Since $L=J^{1/4}$, the error term comes to $O(J^{7/4})$. If $\lambda\geq 2/t$, then $\Sigma_1=O(J^{7/4})$.
		
		We now wish to prove Lemmas~\ref{L:Sigma1InnerSum}, \ref{L:Sigma1Integral}, and \ref{L:Sigma1OuterSum}. To do this, we use the following three results from Montgomery's monograph~\cite{Mon1994}.
		\begin{lemma}[{\cite[Chapter 1, inequality~(13)]{Mon1994}}]\label{L:MontgomerySumIntegral}
			Let $(x_n)_{n=1}^N$ be a finite sequence with discrepancy $D(N)$. Suppose that $F$ is of bounded variation on $[0,1]$ and that $F$ is continuous at the points $x_n$. Then,
			\begin{equation*}
				\Bigg|
					\sum_{n=1}^N F(x_n) - N\int_0^1 F(x)\,\dv x
				\Bigg|
				\leq \frac{1}{2}D(N)\Var_{[0,1]}(F).
			\end{equation*}
		\end{lemma}
		
			\begin{lemma}[{Erd\H{o}s-Tur\'an Inequality~\cite[Corollary 1.1]{Mon1994}}]\label{L:ErdosTuran}
				Let $(x_n)_{n=1}^N$ be a finite sequence with discrepancy $D(N)$. For any positive integer $K$,
				\begin{equation*}
					D(N) \leq \frac{N}{K+1} + 3\sum_{l=1}^K \frac{1}{l} \,
					\Bigg|
						\sum_{n\leq N} e(lx_n)
					\Bigg|.
				\end{equation*}
			\end{lemma}
			
			\begin{lemma}[{\cite[Chapter 3, Theorem~10~(Process B)]{Mon1994}}]\label{L:MontgomeryBoundExpSum}
				Let $A$ be a positive absolute constant. Suppose that $f(x)$ is a real-valued function such that $0<\lambda_2\leq f''(x) \leq A\lambda_2$ for all $x\in[a.b]$, and suppose that $\abs{f^{(3)}(x)}\leq A\lambda_2 (b-a)^{-1}$ and that $\abs{f^{(4)}(x)}\leq A\lambda_2 (b-a)^{-2}$ throughout this interval. Put $f'(a)=\alpha$, $f'(b)=\beta$. For integers $\nu\in[\alpha,\beta]$, let $x_\nu$ be the root of the equation $f'(x)=\nu$. Then
				\begin{equation}\label{eqLemmaMontgomery}
					\sum_{a\leq n\leq b} e(f(n)) = e(1/8)\sum_{\alpha\leq\nu\leq\beta} \frac{e(f(x_\nu)-\nu x_nu)}{\sqrt{f''(x_\nu)}} + O(\log(2+\beta-\alpha)) + O(\lambda_2^{-1/2}).
				\end{equation}
			\end{lemma}
			
			\begin{proof}[Proof of Lemma~\ref{L:Sigma1InnerSum}]
				We will use Lemma~\ref{L:MontgomeryBoundExpSum} to find an upper bound for the exponential sum
				\begin{equation*}
					\sum_{\frac{iJ}{L}\leq n \leq \frac{(i+1)J}{L}} e\left(\frac{ltJ^2}{n}\right),
				\end{equation*}
				with which we can apply Lemma~\ref{L:ErdosTuran} to bound the discrepancy $D(N)$ (where $N=J/L+O(1)$) of the sequence $x_n=tJ^2/n$. In particular, we will show
				\begin{equation}\label{E:DiscrepancyBound}
					D(N) \ll \frac{J^{2/3}}{L} + J^{1/2}.
				\end{equation}
				Since the function $F(x)=\mu\big(A_{R,S,y}(x)\big)$ is of bounded variation over $[0,1]$, the bound \eqref{E:DiscrepancyBound} and Lemma~\ref{L:MontgomerySumIntegral} give us Lemma~\ref{L:Sigma1InnerSum}. We now wish to prove \eqref{E:DiscrepancyBound}.
				First we apply Lemma~\ref{L:MontgomeryBoundExpSum} with
				\begin{equation*}
					f(x) = \frac{ltJ^2}{x}, \quad
					a = \frac{iJ}{L}, \quad
					b = \frac{(i+1)J}{L}, \quad
					\lambda_2 = \frac{2ltL^3}{(i+1)^3J},\quad
					A = 96.
				\end{equation*}    
				One thus finds,
				\begin{equation*}
					f'(x) = -\frac{ltJ^2}{x^2},\quad 
					f''(x) = \frac{2ltJ^2}{x^3}, \quad
					f^{(3)}(x) = -\frac{6ltJ^2}{x^4}, \quad
					f^{(4)}(x) = \frac{24ltJ^2}{x^5}, 
				\end{equation*}
				and
				\begin{equation*}
					(b-a)^{-1} = \frac{L}{J}, \quad
					\alpha = f'(a) = -\frac{ltL^2}{i^2}, \quad 
					\beta = f'(b) = -\frac{ltL^2}{(i+1)^2}.
				\end{equation*}
				We see that the conditions for Lemma~\ref{L:MontgomeryBoundExpSum} 
				are satisfied, since
				\begin{gather*}
					f''(x) \geq f''(b) = \frac{2ltL^3}{(i+1)^3J} = \lambda_2, \\
					f''(x) \leq f''(a) = \frac{2ltL^3}{i^3J} = \lambda_2 \frac{(i+1)^3}{i^3} \leq 8 \lambda_2 \leq A\lambda_2, \\
					\abs{f^{(3)}(x)} \leq \abs{f^{(3)}(a)} = \frac{6ltL^4}{i^4J^2} = 3\lambda_2 (b-a)^{-1}\frac{(i+1)^3}{i^4} \leq 24 \lambda_2 (b-a)^{-1} \leq A\lambda_2 (b-a)^{-1}, \\
					\abs{f^{(4)}(x)} \leq \abs{f^{(4)}(a)} = \frac{24ltL^5}{i^5 J^3} = 12 \lambda_2 (b-a)^{-2} \frac{(i+1)^3}{i^5} \leq 96\lambda_2 (b-a)^{-2} = A\lambda_2 (b-a)^{-2}.
				\end{gather*}
				
				\medskip
				Now, we want to bound the right-hand side of 
				relation~\eqref{eqLemmaMontgomery}. 
				Firstly,
				\begin{equation*}
					\lambda_2^{-1/2} = \frac{(i+1)^{3/2} \sqrt{J}}{\sqrt{2t}L^{3/2}} \frac{1}{\sqrt{l}} \ll \sqrt{\frac{J}{l}}.
				\end{equation*}
				Also, since
				\begin{equation*}
					\beta-\alpha 
					= ltL^2\left(\frac{1}{i^2}-\frac{1}{(i+1)^2}\right) 
					= ltL^2\left(\frac{2i+1}{i^2(i+1)^2}\right) 
					\ll \frac{l}{L},
				\end{equation*}
				we find
				\begin{gather*}
					\log(2+\beta-\alpha) = \log\left(2+\frac{2ltL^2}{i(i+1)^2}\right) \ll \log\left(\frac{l}{L}\right) + 1, \\[3pt]
					\intertext{and}       
					e(1/8)\sum_{\alpha\leq\nu\leq\beta} \frac{e(f(x_\nu)-\nu x_\nu)}{\sqrt{f''(x_\nu)}} \ll \left(1 + \sdfrac{l}{L}\right)\lambda_2^{-1/2} \ll \sqrt{\sdfrac{J}{l}} + \sdfrac{\sqrt{Jl}}{L}\,.
				\end{gather*}
				So, finally,
				\begin{equation*}
					\sum_{\frac{iJ}{L}\leq n\leq \frac{(i+1)J}{L}} e\left(\frac{ltJ^2}{n}\right) \ll \sqrt{\sdfrac{J}{l}} + \sdfrac{\sqrt{Jl}}{L} + \log\Big(\sdfrac{l}{L}\Big) + 1
				\end{equation*}
				
				We now plug this into the Erd\H{o}s-Tur\'an inequality using the sequence 
				$\Big\{\Ldfrac{tJ^2}{n}\Big\}_{n=iJ/L}^{(i+1)J/L}$ and recalling that $L=J^{1/4}$, we find that the discrepancy is bounded as follows:
				\begin{align*}
					D(N) &\ll \frac{J}{L} \frac{1}{K+1} + \sum_{l=1}^{K} \frac{1}{l} \left(\sqrt{\frac{J}{l}} + \frac{\sqrt{Jl}}{L} + 1 + \log\left(\sdfrac{l}{L}\right)\right) \\
					&\ll \frac{J}{L}\frac{1}{K} + \sum_{l=1}^K \frac{\sqrt{J}}{l^{3/2}} + \frac{\sqrt{J}}{L\sqrt{l}} + \frac{1}{l} + \frac{1}{l}\log\left(\sdfrac{l}{L}\right) \\
					&\ll \frac{J}{L} \frac{1}{K} + \sqrt{J} + \frac{\sqrt{JK}}{L} + \log^2(K) + \log(L) \log(K).
				\end{align*}
				Taking $K=J^{1/3}$, we find \eqref{E:DiscrepancyBound}, which completes the proof of Lemma~\ref{L:Sigma1InnerSum}.
			\end{proof}
			\medskip
			
			Now, we wish to evaluate the integral arising from Lemma~\ref{L:Sigma1InnerSum}.
			\begin{proof}[Proof of Lemma~\ref{L:Sigma1Integral}]
				For every $x,y\in [0,1]$ consider the set $B_{x,y}:=\big\{ u\in {\mathbb R}: \{x+\{ u\}\}\geq y\big\}$.
				Denote by $\chi_B$ the characteristic function of a set $B\subseteq {\mathbb R}$.
				Tonelli's theorem yields
				\begin{equation*}
					\int_0^1 \mu (A_{R,S,y}(x)) \, \dv x =
					\int_0^1 \int_R^S \chi_{B_{x,y}} (u)\, \dv u \dv x=
					\int_R^S \int_0^1 \chi_{B_{x,y}}(u)\, \dv x \dv u .
				\end{equation*}
				
				Then we have
				\begin{equation*}
					\begin{aligned}
						\chi_{B_{x,y}}(u) & = \begin{cases}
							1 & \mbox{\rm if $\{ x+\{ u\}\} \geq y$} \\[1mm]
							0 & \mbox{\rm if $\{ x+\{ u\}\} <y$}
						\end{cases}\\
						& =\begin{cases}
							1 & \mbox{\rm if $y>\{ u\}$ and $y-\{ u\} <x< 1-\{ u\}$} \\[1mm]
							1 & \mbox{\rm if $y< \{ u\}$ and $(0<x< 1-\{ u\} \ \text{or}\ 1-\{ u\}+y <x<1)$} \\[1mm]
							0 & \mbox{\rm else}
						\end{cases} \\
						& = \begin{cases}
							\chi_{[y-\{ u\},1-\{ u\}]} (x) & \mbox{\rm if $1\geq y>\{ u\}$} \\[1mm]
							\chi_{[0,1-\{ u\}] \cup [1-\{ u\}+y,1]} & \mbox{\rm if $0<y<\{ u\}$},
						\end{cases}
					\end{aligned}
				\end{equation*}
				which entails
				\begin{equation*}
					\int_0^1 \chi_{B_{x,y}} (u)\, \dv x = 
					\begin{cases}
						\mu ([y-\{ u\},1-\{ u\}])  & \mbox{\rm if $1\geq y>\{ u\}$} \\[1mm]
						\mu ([0,1-\{ u\}] +\mu ([1-\{ u\}+y,1]) & \mbox{\rm if $0<y< \{ u\}$}
					\end{cases}
					=1-y.
				\end{equation*}
				We obtain
				\begin{equation*}
					\int_0^1 \mu (A_{R,S,y}(x))\, \dv x =\int_R^S (1-y)\, \dv u = (S-R)(1-y).
				\end{equation*}
			\end{proof}
			
			\begin{proof}[Proof of Lemma~\ref{L:Sigma1OuterSum}]
				We once again employ Lemma~\ref{L:MontgomerySumIntegral}. Write
				\begin{equation*}
					f(w) = 2-\frac{t}{w}(1+\lambda)+\frac{\lambda t^2}{2w^2}, \ \ 
					\text{ and }\ \ 
					u_i = \frac{i/L-C_{\lambda, t}}{1-C_{\lambda, t}}.
				\end{equation*}
				Then, we have
				\begin{align*}
					&\phantom{=}\ 
					\Bigg|
						\sum_{C_{\lambda, t}\leq\sdfrac{i}{L}\leq L} f\left(\sdfrac{i}{L}\right) - L\int_{C_{\lambda, t}}^1 f(w)\,\dv w
					\Bigg|
					\\
					&= 
					\Bigg|
						\sum_{0\leq u_i\leq 1} f\left(u_i(1-C_{\lambda, t})+C_{\lambda, t}\right) - L(1-C_{\lambda, t})\int_0^1 f\left(v(1-C_{\lambda, t})+C_{\lambda t}\right)\,\dv v
					\Bigg|
					\\
					&= 
					\Bigg|
						\sum_{0\leq u_i\leq 1} f\left(u_i(1-C_{\lambda, t})+C_{\lambda, t}\right) - \#\{i\in\Z \colon C_{\lambda, t}L\leq i\leq L\} \int_0^1 f\left(v(1-C_{\lambda, t})+C_{\lambda t}\right)\,\dv v
					\Bigg|
					+ O(1) \\[8pt]
					&\ll \Var_{[0,1]}\left(f\left(v(1-C_{\lambda, t})+C_{\lambda, t}\right)\right) D(L) + 1 \\[8pt]
					&= \Var_{[C_{\lambda, t},1]}(f) D(L) + 1,
				\end{align*}
				where $D(L)$ is the discrepancy of the sequence $\{u_i\}\subseteq[0,1]$ (for $i\in\Z\cap[M_{k,r}^-L, M_{k,r}^+L]$). 
				Since $D(L)$ is~$O(1)$ and the function $f$ is of bounded variation, 
				we find that the final line is $O(1)$, which concludes the proof of the lemma.
			\end{proof}
			
			\subsection{\texorpdfstring{Calculating $\Sigma_{-1}$ -- Interference from the Line Below}{Calculating Sigma-1 -- Interference from the Line Below}}\label{SS:Sigma-1}
			We find $\Sigma_{-1}$ in a similar manner to~$\Sigma_1$. Firstly, to find the angle between $Q=(a,m)$ and $\widehat{Q}$, we consider the projection $\tilde{Q}=(s,m-1)$ of $Q$ onto the line $y=m-1$. The distance between $\tilde{Q}$ and the next integer point to its left will be
			\begin{equation*}
				\{s\} = \Big\{-\sdfrac{tJ^2+a}{m}\Big\},
			\end{equation*}
			hence the angle between $Q$ and $\widehat{Q}$ will be
			\begin{equation*}
				\min\left\{\Big\{-\sdfrac{tJ^2+a}{m}\Big\}\sdfrac{m-1}{t^2J^4}+O\left(\sdfrac{1}{J^4}\right),\
				\sdfrac{m}{t^2J^4}
				+O\left(\sdfrac{1}{J^4}\right)\right\} 
				= \Big\{-\sdfrac{tJ^2+a}{m}\Big\}\sdfrac{m}{t^2J^4}
				+O\left(\sdfrac{1}{J^4}\right).
			\end{equation*}
			Therefore, where $x_m$ and $y_m$ are the same as in the previous subsection, 
			we see that $\Sigma_{-1}$ is
			\begin{align*}
				&= \sum_{C_{\lambda, t}J+O(1)\leq m\leq J} 
				\#\left\{\sdfrac{a}{m}\in\sdfrac{\Z}{m}\cap
				\left[\sdfrac{tJ^2}{m^2}-\sdfrac{J}{m}
				+O\left(\sdfrac{1}{J}\right),\, \sdfrac{J}{m}\right] 
				:
				\Big\{-x_m-\Big\{\sdfrac{a}{m}\Big\}\Big\}\geq y_m 
				+ O\left(\sdfrac{1}{J}\right)\right\} \\
				&= \sum_{C_{\lambda, t}J\leq m\leq J} 
				m\mu\left(\left\{u\in\left[-\sdfrac{J}{m},\, \sdfrac{J}{m}-\sdfrac{tJ^2}{m^2}\right] 
				:
				\{-x_m+u\} \geq y_m\right\}\right)
				+ O(J).
			\end{align*}
			Just as in the previous subsection, let
			\begin{gather*}
				L=J^{1/4}, \\
				I_i = \left[\sdfrac{iJ}{L},\, \sdfrac{(i+1)J}{L}\right], \\
				R_m = -\frac{J}{m}, \\[1pt]
				S_m = \frac{J}{m}-\frac{tJ^2}{m^2}, \\[5pt]
				A_{R,S,y}(x) = \{u\in[R,S] 
				:
				\{x+\{u\}\}\geq y\}.
			\end{gather*}
			Then,
			\begin{equation}
				\Sigma_{-1} 
				= \sum_{C_{\lambda, t}L\leq i\leq L} \sum_{m\in I_i} m\mu\big(A_{R_m, S_m, y_m}(-x_m)\big) 
				+ O\left(\sdfrac{J^2}{L}\right).
			\end{equation}
			We can easily adapt Lemma~\ref{L:Sigma1InnerSum} to see that
			\begin{equation*}
				\abs{
					\sum_{m\in I_i} \mu\big(A_{R,S,y}(-x_m)\big) - \#I_i\int_0^1 
					\mu\big(A_{R,S,y}(-x)\big)\,\dv x
				} \ll \frac{J^{2/3}}{L} + J^{1/2}.
			\end{equation*}
			Furthermore, since $\mu(A_{R,S,y}(x))$ is a $1$-periodic function, we find
			\begin{equation*}
				\int_0^1 \mu\big(A_{R,S,y}(-x)\big)\,\dv x 
				= \int_{-1}^0 \mu\big(A_{R,S,y}(x)\big)\,\dv x 
				= \int_0^1 \mu\big(A_{R,S,y}(x)\big)\,\dv x,
			\end{equation*}
			hence,
			\begin{align*}
				\Sigma_{-1} &= \sum_{C_{\lambda, t}L\leq i\leq L}\#I_i\int_0^1 \mu\big(A_{R_m, S_m, y_m}(x)\big)\,\dv x + O\left(\sdfrac{J^2}{L} + J^{5/3} + J^{3/2}L\right) \\
				&= \Sigma_1 + O\Big(\sdfrac{J^2}{L} + J^{5/3} + J^{3/2}L\Big).
			\end{align*}
			Therefore, since $L=J^{1/4}$,
			\begin{equation*}
				\Sigma_{-1} 
				= \Sigma_1 + O\big(J^{7/4}\big) 
				= 2J^2(1-C_{\lambda, t}) 
				-tJ^2(1+\lambda)\log\Big(\sdfrac{1}{C_{\lambda, t}}\Big)
				-\sdfrac{\lambda t^2 J^2}{2}
				\left(1-\sdfrac{1}{C_{\lambda, t}}\right) + O\big(J^{7/4}\big).
			\end{equation*}
			
			\subsection{\texorpdfstring{Calculating $\Sigma_0$ -- No Interference}{Calculating Sigma0 -- No Interference}}\label{SS:Sigma0}
			Here, there is no interference from adjacent lines, so we deal with this similarly to how we dealt with the case when $t>2$ in Section~\ref{S:tBiggerThan2}. We want to count the points $(a,m)$ satisfying
			\begin{equation*}
				m\geq \frac{\lambda t J}{2} + O(1).
			\end{equation*}
			Thus, if $\lambda \leq 2/t$,
			\begin{equation*}
				\begin{split}
					\Sigma_0 
					&= \sum_{\frac{\lambda tJ}{2}\leq m\leq J} 
					\#\big(\Z\cap[A_m, B_m]\cap[-J,\, J]\big) + O(J) \\
					&= \sum_{\frac{\lambda tJ}{2}\leq m\leq J}
					\#\left(\Z\cap\left[J-\frac{tJ^2}{m}+O(1),\;
					\frac{tJ^2}{m}-J+O(1)\right]\cap[-J,J]\right) 
					+O(J) \\ 
					&= \sum_{\frac{\lambda tJ}{2}\leq m\leq J}
					\#\left(\Z\cap\left[J-\frac{tJ^2}{m},\; \frac{tJ^2}{m}-J\right]\cap[-J,J]\right) 
					+O(J).
				\end{split}
			\end{equation*}
			Since $tJ^2/m \leq 2J$ when $m\geq tJ/2$, we split into the cases when $\lambda\leq1$, $1\leq\lambda\leq 2/t$, and $\lambda\geq 2/t$. When $\lambda\leq 1$,
			\begin{align*}
				\Sigma_0 &= \sum_{\frac{\lambda t J}{2} \leq m \leq \frac{tJ}{2}} 2J + 2\sum_{\frac{tJ}{2}\leq m\leq J} \left(\frac{tJ^2}{m}-J\right) + O(J) \\
				&= tJ^2(1-\lambda) + 2tJ^2\log\left(\frac{2}{t}\right) - 2J^2\left(1-\frac{t}{2}\right) + O(J).
			\end{align*}
			When $\lambda\geq 1$,
			\begin{align*}
				\Sigma_0 &= 2\sum_{\frac{\lambda tJ}{2}\leq m\leq J} 
				\left(\frac{tJ^2}{m}-J\right) + O(J) \\
				&= 2tJ^2\log\left(\frac{2}{\lambda t}\right) - 2J^2\left(1-\frac{\lambda t}{2}\right) + O(J).
			\end{align*}
			When $\lambda\geq2/t$, we simply have $\Sigma_0 = O(J)$.
			
			\subsection{Putting it all together}\label{SS:tBetween1And2Final}
			Firstly, since each $\Sigma_j = O(J^2)$,
			\begin{equation*}
				G_{t,J}(\lambda) = \frac{1}{2J^2+O(J)}(\Sigma_1+\Sigma_0+\Sigma_{-1}) = \frac{1}{2J^2}(\Sigma_1+\Sigma_0+\Sigma_{-1}) + O(J).
			\end{equation*}
			Therefore, summing our results from the previous 3 subsections, keeping in mind that
			\begin{equation*}
				C_{\lambda, t} = \begin{cases}
					\frac t2, &\text{if } \lambda\leq 1; \\[6pt]
					\frac{ \lambda t}{2}, &\text{if } \lambda \geq 1;
				\end{cases}
			\end{equation*}
			we find
			\begin{equation*}
				G_{t,J}(\lambda) = \begin{cases}
					1+\frac{\lambda t}{2} + \lambda t\log\left(\frac{t}{2}\right) - \frac{\lambda t^2}{2} + O\left(J^{-1/4}\right), &\text{if } 0\leq\lambda\leq 1; \\[6pt]
					1-\frac{\lambda t}{2} + \lambda t\log\left(\frac{\lambda t}{2}\right) - \frac{\lambda t^2}{2}+t + O\left(J^{-1/4}\right), &\text{if } 1\leq\lambda\leq\frac{2}{t}; \\[6pt]
					O\left(J^{-1/4}\right), &\text{if } \lambda\geq\frac{2}{t}.
				\end{cases}
			\end{equation*}
			Finally, taking $J\to\infty$, we get the second part of Theorem~\ref{T:GLambdaClosedForm}.
			
			\section{\texorpdfstring{All values of $t$}{All values of t}}\label{S:AllValuesOft}
			
			We now wish to prove the existence of $G_t(\lambda)$ for all values of $t>0$, and in particular show that it is given by the expression in Theorem~\ref{T:GLambdaAllt}. Firstly, for any $t$, up to how many lines away can we get interference? Consider the line passing through $(-J,m)$ and $(J,m+h)$,
			\begin{equation*}
				x=\frac{2J}{h}(y-m)-J.
			\end{equation*}
			This intersects the $x$-axis at $(-2Jm/h-J,0)$. Since $m$ can be as large as $J$, we get interference from~$h$ lines away when
			\begin{equation*}
				t \leq \frac{2}{h}+\frac{1}{J}.
			\end{equation*}
			In other words, for any $t$ satisfying
			\begin{equation*}
				\frac{2}{h+1} < t \leq \frac{2}{h},
			\end{equation*}
			we can take $J$ sufficiently large such that we get interference from up to $h:=\floor{2/t}$ lines (but no more) away.
			For any point $Q=(a,m)$, the next point seen by $P_{t,J}$ could be on lines $y=m-k$ to $y=m+r$ (but no more), where $k,r \leq h$ (the exact values of $k$ and $r$ depend on the point $(a,m)$). Let $\tilde{Q}_j$ be the projection of the point $Q$ onto the line $y=m+j$. 
			Then, the distance between $\tilde{Q}_j$ 
			and the first integer point to its left will be the fractional part
			\begin{equation*}
				\Big\{j\Big(\sdfrac{tJ^2+a}{m}\Big)\Big\}.
			\end{equation*}
			So, we want to find points that satisfy
			\begin{equation*}
				\min\bigg\{\sdfrac{m}{t^2J^4}
				+O\Big(\sdfrac{1}{J^4}\Big),\,
				\min_{\substack{-k\leq j \leq r \\ j\neq 0}} 
				\Big\{j\Big\{\sdfrac{tJ^2+a}{m}\Big\}\sdfrac{m+j}{t^2 J^4}
				+O\Big(\sdfrac{1}{J^4}\Big)\Big\}\bigg\} 
				\geq \frac{\lambda}{2tJ^3} 
				+ O\Big(\sdfrac{1}{J^4}\Big).
			\end{equation*}
			
			Similar to \eqref{E:SimplifyMinsAngles}, we see that this is equivalent to saying that
			\begin{equation*}
				\min_{\substack{-k\leq j\leq r \\ j\neq 0}} 
				\Big\{j\Big(\sdfrac{tJ^2+a}{m}\Big)\Big\} \geq \sdfrac{\lambda t J}{2m} + O\left(\sdfrac{1}{J}\right),
			\end{equation*}
			where it is understood that this is a minimum over fractional parts.
			We will split the rectangle up into regions corresponding to pairs $(k,r)$ where in such regions we can have interference from lines $y=m-k$ to $y=m+r$ but no more (see Figure~\ref{krRegions}). Firstly, a point $(a,m)$ gets interference from the line $y=m+j$ when
			\begin{equation*}
				a\leq J-\frac{j}{m}tJ^2,
			\end{equation*}
			so there is interference from up to line $y=m+r$, but not line $y=m+r+1$ or above when
			\begin{equation*}
				J-\frac{(r+1)}{m}tJ^2 \leq a \leq J-\frac{r}{m}tJ^2.
			\end{equation*}
			Similarly, one finds that we have interference from down to line $y=m-k$, but not line $y=m-k-1$ or below when
			\begin{equation*}
				\frac{k}{m}tJ^2-J\leq a\leq \frac{(k+1)}{m}tJ^2 - J.
			\end{equation*}
			Therefore, for any $m$, the point $(a,m)$ will have interference only between lines $y=m-k$ and $y=m+r$ when
			\begin{equation*}
				\max\left\{\frac{k}{m}tJ^2-J, J-\frac{(r+1)}{m}tJ^2\right\}
				\leq a
				\leq \min\left\{\frac{(k+1)}{m}tJ^2-J, J-\frac{r}{m}tJ^2\right\}.
			\end{equation*}
			
			Now, to find an inequality for $m$, we just check for which values of $m$ the left side of the above inequality is less than the right side. We have for all $m$ that
			\begin{equation*}
				\begin{split}
					\frac{k}{m}tJ^2-J &\leq \frac{(k+1)}{m}tJ^2 - J, \\[6pt]
					J-\frac{(r+1)}{m}tJ^2 &\leq J-\frac{r}{m}tJ^2,
				\end{split}
			\end{equation*}
			so we need only verify
			\begin{equation*}
				\begin{split}
					\frac{k}{m}tJ^2 - J &\leq J-\frac{r}{m}tJ^2, \\[6pt]
					J-\frac{(r+1)}{m}tJ^2 &\leq \frac{(k+1)}{m}tJ^2 - J.
				\end{split}
			\end{equation*}
			We see from these two inequalities that we want
			\begin{equation*}
				(k+r)\frac{tJ}{2} \leq m \leq (k+r+2)\frac{tJ}{2}.
			\end{equation*}
			We further restrict $m$ by requiring $m\leq J$ so that we stay inside the rectangle, and $m\geq \frac{\lambda tJ}{2}$, otherwise all angles on the line $y=m$ will be too small. So, if we define
			\begin{equation*}
				M_{k,r}^- := \max\left\{\sdfrac{\lambda t}{2},\, \sdfrac{(k+r)t}{2}\right\}, 
				\quad
				M_{k,r}^+ := \min\left\{1,\, \sdfrac{(k+r+2)t}{2}\right\};
			\end{equation*}
			\begin{equation*}
				R_{k,r,m} := \max\left\{\sdfrac{ktJ^2}{m^2}-\sdfrac{J}{m},\, \sdfrac{J}{m}-\sdfrac{(r+1)tJ^2}{m^2}\right\},\quad
				S_{k,r,m} := \min\left\{\sdfrac{(k+1)tJ^2}{m^2}-\frac{J}{m},\, \sdfrac{J}{m}-\sdfrac{rtJ^2}{m^2}\right\};
			\end{equation*}
			and
			\begin{equation*}
				x_m := \sdfrac{tJ^2}{m}, \quad
				y_m := \sdfrac{\lambda tJ}{2m},
			\end{equation*}
			then $(2J^2+O(J))G_{t,J}(\lambda)$ is
			\begin{equation*}
				\begin{split}
					=\sum_{\substack{0\leq k,r\leq h\\ M_{k,r}^- \leq M_{k,r}^+}} \sum_{M_{k,r}^- J\leq m \leq M_{k,r}^+ J}  
					\#\bigg\{\sdfrac{a}{m}\in \sdfrac{\Z}{m}\cap[R_{k,r,m},\ S_{k,r,m}] 
					:
					\min_{\substack{-k\leq j \leq r \\ j\neq 0}} \left\{j\left(x_m + \sdfrac{a}{m}\right)\right\} \geq y_m\bigg\} + O(J). 
				\end{split}  
			\end{equation*}
			
			\begin{figure}[ht]
				\centering    
				\includegraphics[width=0.74\textwidth]{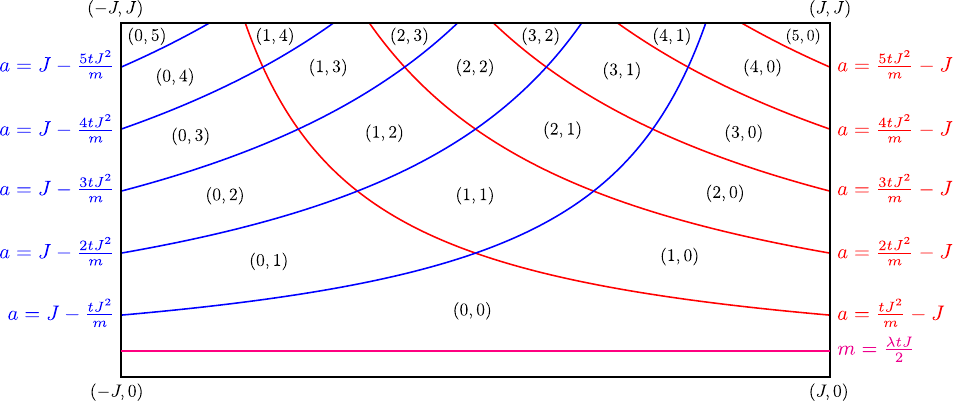}
				\caption{A diagram of the rectangle split into regions of interference for $t=0.35$. Inside each region the pair $(k,r)$ associated to it is written.}
				\label{krRegions}
			\end{figure}
			
			As we did when looking at the case $1<t\leq 2$, we break up our sum over $m$ into intervals of the form
			\begin{equation*}
				I_i := \left[\frac{iJ}{L},\, \frac{(i+1)J}{L}\right),
			\end{equation*}
			where $L=J^{1/4}$.
			Therefore, we now have that $(2J^2+O(J))G_{t,J}(\lambda)$ is
			\begin{equation*}
				\begin{split}
					=\sum_{\substack{0\leq k,r\leq h\\ M_{k,r}^- \leq M_{k,r}^+}} 
					\sum_{M_{k,r}^- L\leq i \leq M_{k,r}^+ L} \sum_{m\in I_i}  
					\#\bigg\{\sdfrac{a}{m}\in \sdfrac{\Z}{m}\cap[R_{k,r,m},\, S_{k,r,m}] 
					: 
					\min_{\substack{-k\leq j \leq r \\ j\neq 0}} 
					\left\{j\left(x_m + \sdfrac{a}{m}\right)\right\} \geq y_m\bigg\}
					+ O\Big(\fdfrac{J^2}{L}\Big).  
				\end{split}
			\end{equation*}
			We now re-index $m$, $y$, $R$, and $S$ by $i/L$ instead of $m$ as
			\begin{gather*}
				m=\sdfrac{iJ}{L}\left(1+O\left(\fdfrac{1}{L}\right)\right), \\
				y_{\frac{i}{L}} := \sdfrac{\lambda tL}{2i} 
				= y_{m}+O\left(\fdfrac{1}{L}\right), \\
				R_{k,r,\frac{i}{L}} := \max\left\{\sdfrac{ktL^2}{i^2} - \sdfrac{L}{i},\; 
				\sdfrac{L}{i}-\sdfrac{(r+1)tL^2}{i^2}\right\} 
				= R_{k,r,m}+O\left(\fdfrac{1}{L}\right), \\
				S_{k,r,\frac{i}{L}} := \min\left\{\sdfrac{(k+1)tL^2}{i^2}
				-\sdfrac{L}{i},\; \sdfrac{L}{i}-\sdfrac{rtL^2}{i^2}\right\}
				= S_{k,r,m}+O\left(\fdfrac{1}{L}\right).
			\end{gather*}
			We first replace the count of a discrete set in our previous expression for $G_{t,J}(\lambda)$ with the measure of a continuous set to find
			that $(2J^2+O(J))G_{t,J}(\lambda)$ is
			\begin{equation*}
				\begin{split}
					=\sum_{\substack{0\leq k,r\leq h\\ M_{k,r}^- \leq M_{k,r}^+}} 
					\sum_{M_{k,r}^- L\leq i \leq M_{k,r}^+ L} \sum_{m\in I_i} 
					m\cdot\mu\Big(\Big\{u\in [R_{k,r,m},\ S_{k,r,m}] :
					\min_{\substack{-k\leq j \leq r \\ j\neq 0}} \left\{j\left(x_m + u\right)\right\} \geq y_m\Big\}\Big)
					+ O\left(\fdfrac{J^2}{L}\right).
				\end{split}
			\end{equation*}
			We now use our re-indexing of $m$, $y$, $R$, and $S$ to see that the above becomes
			\begin{equation*}
				\begin{split}
					= &\sum_{\substack{0\leq k,r\leq h\\ M_{k,r}^- \leq M_{k,r}^+}} 
					\sum_{M_{k,r}^- L\leq i \leq M_{k,r}^+ L}
					\fdfrac{iJ}{L}\Big(1+O\Big(\fdfrac{1}{L}\Big)\Big)
					\sum_{m\in I_i} \\
					&\quad
					\mu\bigg(\Big\{u\in \left[R_{k,r,\frac{i}{L}}+O\left(\fdfrac{1}{L}\right),\, S_{k,r,\frac{i}{L}}+O\left(\fdfrac{1}{L}\right)\right] 
					:
					\min_{\substack{-k\leq j \leq r \\ j\neq 0}} 
					\big\{j\left(x_m + u\right)\big\} \geq y_{\frac{i}{L}}
					+O\left(\fdfrac{1}{L}\right)\Big\}\bigg) 
					+ O\Big(\fdfrac{J^2}{L}\Big).
				\end{split}
			\end{equation*}
			In the end, we sort out the error terms to obtain that 
			$(2J^2+O(J))G_{t,J}(\lambda)$ is
			\begin{equation*}
				\begin{split}
					&= \sum_{\substack{0\leq k,r\leq h\\ M_{k,r}^- \leq M_{k,r}^+}} 
					\sum_{M_{k,r}^- L\leq i \leq M_{k,r}^+ L}
					\sdfrac{iJ}{L}\sum_{m\in I_i}  
					\mu\bigg(\Big\{
					u\in \big[R_{k,r,\frac{i}{L}},\ S_{k,r,\frac{i}{L}}\big] 
					:
					\min_{\substack{-k\leq j \leq r \\ j\neq 0}} 
					\big\{j\left(x_m + u\right)\big\} \geq y_{\ldfrac{i}{L}}\Big\}\bigg) 
					+ O\Big(\fdfrac{J^2}{L}\Big) \\
					&=\sum_{\substack{0\leq k,r\leq h\\ M_{k,r}^- \leq M_{k,r}^+}}
					\sum_{M_{k,r}^- L\leq i \leq M_{k,r}^+ L}
					\sdfrac{iJ}{L}\sum_{m\in I_i}
					\mu\big(A_{k,r,\frac{i}{L}}(x_m)\big)
					+O\Big(\fdfrac{J^2}{L}\Big),
				\end{split}
			\end{equation*}
			where
			\begin{equation*}
				A_{k,r,\frac{i}{L}}(x) 
				:= \Big\{z\in \big[R_{k,r,\frac{i}{L}}+x,\ S_{k,r,\frac{i}{L}}+x\big] 
				:
				\min_{\substack{-k\leq j \leq r \\ j\neq 0}} \left\{jz\right\} \geq y_{\frac{i}{L}}\Big\}.
			\end{equation*}
			We handle the inner sum using the following lemma.
			\begin{lemma}\label{L:AlltInnerSum}
				Let
				\begin{equation*}
					A_{k,r,R,S,y}(x) 
					:= \Big\{z\in [R+x,\ S+x] 
					:
					\min_{\substack{-k\leq j \leq r \\ j\neq 0}} \left\{jz\right\} \geq y\Big\}.
				\end{equation*}
				Then,
				\begin{equation*}
					\abs{\sum_{m\in I_i} \mu\big(A_{k,r,R,S,y}(x_m)\big) - 
						\#I_i \int_0^1 \mu\big(A_{k,r,R,S,y}(x)\big)\,\dv x} 
					\ll \frac{J^{2/3}}{L} + J^{1/2}.
				\end{equation*}
			\end{lemma}
			From Lemma~\ref{L:AlltInnerSum}, we find
			that $(2J^2+O(J))G_{t,J}(\lambda)$ is
			\begin{equation*}
				\begin{split}
					&= \sum_{\substack{0\leq k,r\leq h\\ M_{k,r}^- \leq M_{k,r}^+}}
					\sum_{M_{k,r}^- L\leq i \leq M_{k,r}^+ L}
					\sdfrac{iJ}{L}\Big(\sdfrac{J}{L}+O(1)\Big)
					\int_0^1 \mu\big(A_{k,r,\frac{i}{L}}(x)\big)\,\dv x 
					+O\left(\sdfrac{J^2}{L}+J^{5/3}+J^{3/2}L\right) \\
					&= \sdfrac{J^2}{L}
					\sum_{\substack{0\leq k,r\leq h\\ M_{k,r}^- \leq M_{k,r}^+}}
					\sum_{M_{k,r}^- L\leq i \leq M_{k,r}^+ L}\frac{i}{L}\int_0^1 \mu\big(A_{k,r,\frac{i}{L}}(x)\big)\,\dv x 
					+O\left(\sdfrac{J^2}{L}+J^{5/3}+J^{3/2}L\right).
				\end{split}
			\end{equation*}
			We now deal with the next sum using the following technical lemma.
			\begin{lemma}\label{L:AlltOuterSum}
				We have
				\begin{equation*}
						\Bigg| 
						\sum_{M_{k,r}^-L\leq i\leq M_{k,r}^+L} \frac{i}{L}
						\int_0^1 \mu\big(A_{k,r,\frac{i}{L}}(x)\big)\,\dv x - L \int_{M_{k,r}^-}^{M_{k,r}^+} w\int_0^1 
						\mu\big(A_{k,r,w}(x)\big)\,\dv x\dv w 
						\Bigg|
					\ll 1.
				\end{equation*}
			\end{lemma}
			From Lemma~\ref{L:AlltOuterSum}, we obtain 
			that $(2J^2+O(J))G_{t,J}(\lambda)$ is
			\begin{equation*}
				\begin{split}
					= J^2\sum_{\substack{0\leq k,r\leq h\\ M_{k,r}^- \leq M_{k,r}^+}} 
					\int_{M_{k,r}^-}^{M_{k,r}^+} w\int_0^1 
					\mu\big(A_{k,r,w}(x)\big)\,\dv x \dv w 
					+O\left(\sdfrac{J^2}{L}+J^{5/3}+J^{3/2}L\right).    
				\end{split}
			\end{equation*}
			Since $L=J^{1/4}$, dividing both sides by $2J^2+O(J)$ yields
			\begin{equation*}
				G_{t,J}(\lambda) = \frac{1}{2}\sum_{\substack{0\leq k,r\leq h\\ M_{k,r}^- \leq M_{k,r}^+}} \int_{M_{k,r}^-}^{M_{k,r}^+} w\int_0^1 
				\mu\big(A_{k,r,w}(x)\big)\,\dv x \dv w + O\big(J^{-1/4}\big).
			\end{equation*}
			Finally, taking $J\to\infty$ we obtain Theorem~\ref{T:GLambdaAllt}.
			
			\begin{proof}[Proof of Lemma~\ref{L:AlltInnerSum}]
				From Lemma~\ref{L:MontgomerySumIntegral}, we see that
				\begin{equation*}
					\abs{\sum_{m\in I_i} \mu\big(A_{k,r,R,S,y}(x_m)\big) - \#I_i\int_0^1 \mu\big(A_{k,r,R,S,y}(x)\big)\,\dv x} \ll D(\#I_i)
					\Var_{[0,1]}\big(\mu(A_{k,r,R,S,y}(x)\big),
				\end{equation*}
				where $D(\#I_i)$ is the discrepancy of the sequence $\{x_m\}_{m\in I_i}$. This is exactly the same as the discrepancy in the proof of Lemma~\ref{L:Sigma1InnerSum}, that is,
				\begin{equation*}
					D(\#I_i) \ll \frac{J^{2/3}}{L} + J^{1/2}.
				\end{equation*}
				Furthermore, the total variation $\Var_{[0,1]}(\mu(A_{k,r,R,S,y}(x))=O(1)$, and consequently we obtain the desired upper bound on the error.
			\end{proof}
			
			\begin{proof}[Proof of Lemma~\ref{L:AlltOuterSum}]
				We again utilize Lemma~\ref{L:MontgomerySumIntegral}. 
				Write
				\begin{align*}
						f(w) & = w\int_0^1 \mu\big(A_{k,r,w}(x)\big)\,\dv x, \\
						\intertext{and}
						u_i &= \frac{\frac{i}{L}-M_{k,r}^-}{M_{k,r}^+-M_{k,r}^-}.
				\end{align*}   
				Then, we have
				\begin{equation*}
					\begin{split}
						&\phantom{=}\ 
							\Bigg|
							\sum_{M_{k,r}^-L\leq i \leq M_{k,r}^+L} f\left(\frac{i}{L}\right) - L\int_{M_{k,r}^-}^{M_{k,r}^+} f(w)\,\dv w
						\Bigg|
						\\
						&=
							\Bigg|\sum_{0\leq u_i \leq 1} f\left(u_i(M_{k,r}^+-M_{k,r}^-)+M_{k,r}^-\right) 
							- L\big(M_{k,r}^+-M_{k,r}^-\big)\int_0^1 
							f\left(v(M_{k,r}^+-M_{k,r}^-)+M_{k,r}^-\right)\,\dv v
						\Bigg|
						\\
						&= \Bigg|\sum_{0\leq u_i \leq 1} f\left(u_i(M_{k,r}^+-M_{k,r}^-)+M_{k,r}^-\right) \\
						&\phantom{=}\  - \#\big\{i\in\Z \colon M_{k,r}^-L\leq i< M_{k,r}^+L\big\}
						\int_0^1 f\left(v(M_{k,r}^+-M_{k,r}^-)+M_{k,r}^-\right)\,\dv v \Bigg| + O(1) \\
						&\ll \Var_{[0,1]}
						\left(f\Big(v\big(M_{k,r}^+-M_{k,r}^-\big)+M_{k,r}^-\Big)\right) D(L) + 1 \\
						&= \Var_{[M_{k,r}^+,M_{k,r}^-]}(f) D(L) + 1,
					\end{split}
				\end{equation*}
				where $D(L)$ is the discrepancy of the sequence $\{u_i\}\subseteq[0,1]$ (for $i\in\Z\cap[M_{k,r}^-L, M_{k,r}^+L]$). Since the discrepancy is $O(1)$, it remains to show the total variation of $f$ in $[M_{k,r}^-, M_{k,r}^+]$ is also $O(1)$.
				
				Recall
				\begin{equation*}
					A_{k,r,w}(x) := \big\{z\in\big[R_{k,r,w}+x,S_{k,r,w}+x\big] 
					:
					F_{k,r}(z)\geq y_w\big\},
				\end{equation*}
				where
				\begin{equation}\label{E:RecallRS} 
					R_{k,r,w} = \max\left\{\sdfrac{kt}{w^2}-\sdfrac{1}{w},\, \sdfrac{1}{w}-\sdfrac{(r+1)t}{w^2}\right\},\quad
					S_{k,r,w} = \min\left\{\sdfrac{(k+1)t}{w^2}-\sdfrac{1}{w},\, \sdfrac{1}{w}-\frac{rt}{w^2}\right\},    
				\end{equation}
				and
				\begin{equation*}
					F_{k,r}(z) = \min_{\substack{-k\leq j\leq r \\ j\neq 0}} \{jz\}, \quad
					y_w = \frac{\lambda t}{2w}.
				\end{equation*}
				Since $F_{k,r}$ is 1-periodic, we see that
				\begin{align*}
					\mu(A_{k,r,w}(x)) &= \mu\Big(A_{k,r,w}(x) \cap \big[R_{k,r,w}+x, \ceiling{R_{k,r,w}+x}\big]\Big) \\
					&\phantom{=} + \mu \Big(A_{k,r,w}(x) \cap \big[\ceiling{R_{k,r,w}+x},\, \floor{S_{k,r,w}+x}\big]\Big) \\ 
					&\phantom{=} + \mu \Big(A_{k,r,w}(x) \cap \big[\floor{S_{k,r,w}+x},\, S_{k,r,w}+x]\big]\Big) \\[3pt]
					&= \mu\big(A_{k,r,w}(x) \cap \big[\{R_{k,r,w}+x\},\, 1\big]\big) \\ 
					&\phantom{=}+ \big(\floor{S_{k,r,w}+x}-\ceiling{R_{k,r,w}+x}\big)
					\mu\big(A_{k,r,w}(x) \cap [0,1]\big) \\
					&\phantom{=}+ \mu\Big(A_{k,r,w}(x) \cap \big[0,\, \{S_{k,r,w}+x\}\big]\Big).
				\end{align*}
				Let's first study $\mu\big(A_{k,r,w}(x)\cap[0,1]\big)$. In other words, we want to find the measure of the set of points $z\in[0,1]$ such that $F_{k,r}(z)\geq y_w$. In order to do this, we wish to better understand the function $F_{k,r}(z)$. When $z=a/q$, where $a,q\in\Z$, we see that $\{\pm qz\}=0$, so at this point the minimum is attained when $j=\pm q$. The minimum only changes if $u=a/q$ where $\abs{q}\leq \max\{k,r\}\leq h$, or when $\{qz\}=\{q'z\}$ where $\abs{q},\abs{q'}\leq h$, that is when
				\begin{equation*}
					z=\frac{a}{q-q'},
				\end{equation*}
				where $a\in \Z$ and $\abs{q},\abs{q'}\leq h$. In other words, the minimum only changes when $z=a/q$ where $q\in\Z\cap[1,2h]$ and $a\in\Z\cap[0,q]$. So, we partition the interval $[0,1]$ by
				\begin{equation*}
					0 = \frac{a_0}{q_0} < \frac{a_1}{q_1} < \cdots < \frac{a_n}{q_n} < \frac{a_{n+1}}{q_{n+1}} = 1,
				\end{equation*}
				where $q_j \leq 2h$, and in each interval $[a_j/q_j, a_{j+1}/q_{j+1}]$ either $F_{k,r}(z) = q_jz-a_j$ or $F_{k,r}(z)= -q_{j+1}z+a_{j+1}$. In such an interval, if $F_{k,r}(z) = q_jz-a_j$, then
				\begin{equation}
					A_{k,r,w}(x)\cap\left[\frac{a_j}{q_j},\, \frac{a_{j+1}}{q_{j+1}}\right] 
					= \begin{cases}
						\left[\Ldfrac{\lambda t}{2wq_{j}} + \Ldfrac{a_j}{q_j},\  \Ldfrac{a_{j+1}}{q_{j+1}}\right], & \text{if }\  \Ldfrac{\lambda t}{2w} \leq \Ldfrac{q_j a_{j+1}}{q_{j+1}}-a_j, \\[9pt]
						\emptyset, & \text{if }\  \Ldfrac{\lambda t}{2w} \geq \Ldfrac{q_j a_{j+1}}{q_{j+1}} - a_j;
					\end{cases} \label{E:MeasureAHatIntervalPlus}
				\end{equation}
				whereas if $F_{k,r}(z) = -q_{j+1}z+a_{j+1}$, then
				\begin{equation}
					A_{k,r,w}(x)\cap\left[\frac{a_j}{q_j},\,\frac{a_{j+1}}{q_{j+1}}\right] 
					= \begin{cases}
						\left[\Ldfrac{a_j}{q_j},\, \Ldfrac{a_{j+1}}{q_{j+1}}-\Ldfrac{\lambda t}{2wq_{j+1}}\right], & \text{if }\ \Ldfrac{\lambda t}{2w} \leq -\Ldfrac{q_{j+1} a_j}{q_j}+a_{j+1}, \\[9pt]
						\emptyset, & \text{if }\ \Ldfrac{\lambda t}{2w} \geq -\Ldfrac{q_{j+1} a_j}{q_j}+a_{j+1}.
					\end{cases} \label{E:MeasureAHatIntervalMinus}
				\end{equation}
				
				\begin{figure}[ht]
					\centering    
					\includegraphics[width=0.74\textwidth]{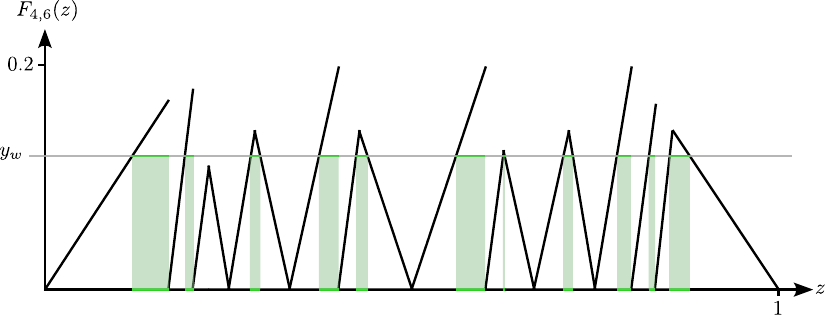}
					\caption{The graph of $F_{4,6}(z)$ for $0\leq z\leq 1$. The green segments represent the subset of $[0,1]$ where $F_{4,6}(z)\geq y_w$.}
					\label{F46z}
				\end{figure}
				
				Summing over all of these sub-intervals, we find that $\mu(A_{k,r,w}(x)\cap[0,1])$ can be written as a piece-wise function over a bounded (depending on $\lambda$ and $t$) number of intervals for $w$. Hence, we can partition the interval $[M_{k,r}^-,M_{k,r}^+]$ by
				\begin{equation*}
					M_{k,r}^- = \gamma_0 < \gamma_1 < \cdots < \gamma_N < \gamma_{N+1} = M_{k,r}^+,
				\end{equation*}
				such that for $w\in[\gamma_j, \gamma_{j+1}]$,
				\begin{equation*}
					\mu\big(A_{k,r,w}(x)\cap[0,1]\big) = \sum_{l=1}^{m_j} \mu(\cI_{j,l}),
				\end{equation*}
				where the $\cI_{j,l}$ are the non-empty intervals in \eqref{E:MeasureAHatIntervalPlus} and \eqref{E:MeasureAHatIntervalMinus}. Firstly, let's find
				\begin{equation*}
					\int_0^1 \big(\floor{S_{k,r,w}+x}-\ceiling{R_{k,r,w}+x}\big)
					\mu\left(\cI_{j,l}\right) \,\dv x.
				\end{equation*}
				If $\{S_{k,r,w}\} \geq \{R_{k,r,w}\}$, then
				\begin{equation*}
					\floor{S_{k,r,w}+x} - \ceiling{R_{k,r,w}+x} = \begin{cases}
						\floor{S_{k,r,w}}-\ceiling{R_{k,r,w}}, &\text{if } 0\leq x\leq 1-\{S_{k,r,w}\}, \\
						\floor{S_{k,r,w}}+1-\ceiling{R_{k,r,w}}, &\text{if } 1-\{S_{k,r,w}\}\leq x \leq 1-\{R_{k,r,w}\}, \\
						\floor{S_{k,r,w}} - \ceiling{R_{k,r,w}}, &\text{if } 1-\{R_{k,r,w}\} \leq x \leq 1.
					\end{cases}
				\end{equation*} 
				In this case, 
				\begin{align*}
					\int_0^1 \big(\floor{S_{k,r,w}+x}-\ceiling{R_{k,r,w}+x}\big)
					\mu\left(\cI_{j,l}\right) \,\dv x
					&= \mu\left(\cI_{j,l}\right)
					\big(\floor{S_{k,r,w}}-\ceiling{R_{k,r,w}}\big) + \mu\left(\cI_{j,l}\right)\big(\{S_{k,r,w}\}-\{R_{k,r,w}\}\big) \\
					&= \mu\left(\cI_{j,l}\right)\big(S_{k,r,w}-R_{k,r,w}-1\big).
				\end{align*}
				If, on the other hand, $\{R_{k,r,w}\} \geq \{S_{k,r,w}\}$, then
				\begin{equation*}
					\floor{S_{k,r,w}+x} - \ceiling{R_{k,r,w}+x} = \begin{cases}
						\floor{S_{k,r,w}}-\ceiling{R_{k,r,w}}, &\text{if } 0\leq x\leq 1-\{R_{k,r,w}\}, \\
						\floor{S_{k,r,w}}-1-\ceiling{R_{k,r,w}}, &\text{if } 1-\{R_{k,r,w}\}\leq x \leq 1-\{S_{k,r,w}\}, \\
						\floor{S_{k,r,w}} - \ceiling{R_{k,r,w}}, &\text{if } 1-\{S_{k,r,w}\} \leq x \leq 1.
					\end{cases}
				\end{equation*}
				We also find in this case that
				\begin{equation*}
					\int_0^1 \big(\floor{S_{k,r,w}+x}-\ceiling{R_{k,r,w}+x}\big)
					\mu\left(\cI_{j,l}\right) \,\dv x
					= \mu\left(\cI_{j,l}\right)\left(S_{k,r,w}-R_{k,r,w}-1\right).
				\end{equation*}
				Next, let us look at
				\begin{equation*}
					\int_0^1 \mu\big(A_{k,r,w}(x) \cap [0, \{S_{k,r,w}+x\}]\big) \,\dv x.
				\end{equation*}
				We start off by noticing that
				\begin{equation*}
					\mu\big(A_{k,r,w}(x) \cap [0, \{S_{k,r,w}+x\}]\big)
					= \sum_{l=1}^{m_j} \mu\big(\cI_{j,l}\cap\left[0, \{S_{k,r,w}+x\}\right]\big).
				\end{equation*}
				Write $\cI_{j,l} = [b_{j,l}, c_{j,l}]$. Then, if $\{S_{k,r,w}\} \leq b_{j,l}$,
				\begin{equation*}
					\mu\left(\cI_{j,l}\cap\left[0, \{S_{k,r,w}+x\}\right]\right) = \begin{cases}
						0, &\text{if }0\leq x\leq b_{j,l}-\{S_{k,r,w}\}, \\[4pt]
						\{S_{k,r,w}\}+x-b_{j,l}, &\text{if }b_{j,l}-\{S_{k,r,w}\}\leq x\leq c_{j,l}-\{S_{k,r,w}\}, \\[4pt]
						c_{j,l}-b_{j,l}, &\text{if }c_{j,l}-\{S_{k,r,w}\} \leq x \leq 1-\{S_{k,r,w}\}, \\[4pt]
						0, &\text{if }1-\{S_{k,r,w}\} \leq x\leq 1.
					\end{cases}
				\end{equation*}
				If, on the other hand, $b_{j,l}\leq\{S_{k,r,w}\}\leq c_{j,l}$, then
				\begin{equation*}
					\mu\left(\cI_{j,l}\cap\left[0, \{S_{k,r,w}+x\}\right]\right) = \begin{cases}
						\{S_{k,r,w}\}+x-b_{j,l}, &\text{if }0\leq x\leq c_{j,l}-\{S_{k,r,w}\}, \\[4pt]
						c_{j,l}-b_{j,l}, &\text{if }c_{j,l}-\{S_{k,r,w}\}\leq x\leq 1-\{S_{k,r,w}\}, \\[4pt]
						0, &\text{if }1-\{S_{k,r,w}\} \leq x \leq 1+b_{j,l}-\{S_{k,r,w}\}, \\[4pt]
						\{S_{k,r,w}\}+x-1-b_{j,l}, &\text{if }1+b_{j,l}-\{S_{k,r,w}\} \leq x\leq 1,
					\end{cases}
				\end{equation*}
				whereas if $\{S_{k,r,w}\}\geq c_{j,l}$, then
				\begin{equation*}
					\mu\left(\cI_{j,l}\cap\left[0, \{S_{k,r,w}+x\}\right]\right) = \begin{cases}
						c_{j,l}-b_{j,l}, &\text{if }0\leq x\leq 1-\{S_{k,r,w}\}, \\[4pt]
						0, &\text{if }1-\{S_{k,r,w}\}\leq x\leq 1+b_{j,l}-\{S_{k,r,w}\}, \\[4pt]
						\{S_{k,r,w}\}+x-1-b_{j,l}, &\text{if }1+b_{j,l}-\{S_{k,r,w}\} \leq x \leq 1+c_{j,l}-\{S_{k,r,w}\}, \\[4pt]
						c_{j,l}-b_{j,l}, &\text{if }1+c_{j,l}-\{S_{k,r,w}\} \leq x\leq 1.
					\end{cases}
				\end{equation*}
				In all three cases, we find
				\begin{equation*}
					\int_0^1 \mu\big(\cI_{j,l}\cap\left[0, \{S_{k,r,w}+x\}\right]\big) \,\dv x
					= \mu(\cI_{j,l}) 
					- \sdfrac{1}{2}\big(c_{j,l}^2-b_{j,l}^2\big).
				\end{equation*}
				Furthermore, from the above we see that
				\begin{equation*}
					\begin{split}
						\int_0^1 \mu\big(\cI_{j,l}\cap\left[\{R_{k,r,w}+x\}, 1\right]\big) \,\dv x
						&= \int_0^1 \mu(\cI_{j,l}) 
						- \mu\big(\cI_{j,l}\cap\left[0,\{R_{k,r,w}+x\}\right]\big) \,\dv x \\
						&= \mu(\cI_{j,l}) - \left[\mu(\cI_{j,l})
						-\sdfrac{1}{2}(c_{j,l}^2-b_{j,l}^2)\right] \\
						&= \frac{1}{2}\big(c_{j,l}^2-b_{j,l}^2\big).
					\end{split}            
				\end{equation*}
				Therefore,
				\begin{equation}
					\label{E:IntegralEqualsIntervalMeasures}
					\begin{split}
						\int_0^1 \mu\big(A_{k,r,w}(x)\big) \,\dv x
						&= \sum_{l=1}^{m_j} \left( \int_0^1 
						\big(\floor{S_{k,r,w}+x}-\ceiling{R_{k,r,w}+x}\big)
						\mu\left(\cI_{j,l}\right) \,\dv x\right. \\
						&\phantom{= \sum_{l=1}^{m_j} \left(\right.}\left.
						\qquad+\int_0^1 
						\mu\big(\cI_{j,l}
						\cap\left[0,\, \{S_{k,r,w}+x\}\right]\big) \,\dv x\right.\\
						&\phantom{= \sum_{l=1}^{m_j} \left(\right.}\left. 
						\qquad\qquad\qquad+\int_0^1 
						\mu\big(\cI_{j,l}
						\cap\left[\{R_{k,r,w}+x\},\, 1\right]\big) \,\dv x \right)
						\\
						&=\sum_{l=1}^{m_j}\left(\mu\left(\cI_{j,l}\right)
						\left(S_{k,r,w}-R_{k,r,w}-1\right) + \mu(\cI_{j,l}) - \sdfrac{1}{2}\big(c_{j,l}^2-b_{j,l}^2\big) + 
						\sdfrac{1}{2}\big(c_{j,l}^2-b_{j,l}^2\big) \right)\\
						&= \sum_{l=1}^{m_j} \mu(\cI_{j,l})(S_{k,r,w}-R_{k,r,w}).
					\end{split}          
				\end{equation}
				From \eqref{E:MeasureAHatIntervalPlus} and \eqref{E:MeasureAHatIntervalMinus}, there exist constants $\alpha_{0,j,l}$ and $\alpha_{1,j,l}$ (depending also on $\lambda$, $t$, $a_j$, $q_j$, $a_{j+1}$, and~$q_{j+1}$) such that
				\begin{equation*}
					\mu(\cI_{j,l}) = \alpha_{0,j,l} + \frac{\alpha_{1,j,l}}{w}.
				\end{equation*}
				By summing over $l$ (noting that $m_j$ depends only on $k$, $r$, and $j$), and inserting 
				\eqref{E:RecallRS},
				we find constants $\alpha'_{0,j}$, $\alpha'_{1,j}$, and $\alpha'_{2,j}$, such that, for $w\in[\gamma_j, \gamma_{j+1}]$,
				\begin{equation*}
					f(w) = w\int_0^1 \mu\big(A_{k,r,w}(x)\big) \,\dv x 
					=  \alpha'_{0,j} + \frac{\alpha'_{1,j}}{w} + \frac{\alpha'_{2,j}}{w^2}.
				\end{equation*}
				Since these constants are bounded (depending on $\lambda$ and $t$), and $N$ is also bounded (depending on $t$), we find
				\begin{align*}
					\Var_{[M_{k,r}^-,M_{k,r}^+]}(f) &= \sum_{j=0}^N \Var_{[\gamma_j,\gamma_{j+1}]}(f) \\
					&= \sum_{j=0}^N \int_{\gamma_j}^{\gamma_{j+1}} \abs{f'(w)}\,\dv w \\
					&= \sum_{j=0}^N \int_{\gamma_j}^{\gamma_{j+1}} 
					\abs{-\sdfrac{\alpha'_{1,j}}{w^2} - \sdfrac{2\alpha'_{2,j}}{w^3}} \,\dv w \\
					&\ll \sum_{j=0}^N \int_{\gamma_j}^{\gamma_{j+1}} 1 \,\dv w \\
					&\ll 1,
				\end{align*}
				and this concludes the proof of the lemma.    
			\end{proof}
			
			\subsection{\texorpdfstring{A formula for $G_t(\lambda)$ and the proof of Theorem~\ref{T:GLambdaAlltConstants}}{A formula for Gt(lambda) and the proof of Theorem~\ref{T:GLambdaAlltConstants}}}\label{SS:ShapeOfGtLambda}
			The proof of Lemma~\ref{L:AlltOuterSum} helps 
			us to better describe the gap distribution function
			$G_t(\lambda)$, whose general `recipe' is given by Theorem~\ref{T:GLambdaAlltConstants}. Firstly, we shall fix an integer $h$ and suppose that $t$ lies in the interval
			\begin{equation*}
				\frac{2}{h+1} < t \leq \frac{2}{h}.
			\end{equation*}
			Then, as we saw above, we have
			\begin{equation*}
				w\int_0^1 \mu\big(A_{k,r,w}(x)\big)\,\dv x 
				= w(S_{k,r,w}-R_{k,r,w})\sum_{l=1}^{n_{k,r}} \mu(\cI_l).
			\end{equation*}
			Furthermore, we also saw that, for each interval $\cI_l$, 
			we have
			\begin{equation*}
				\mu(\cI_l) 
				= \begin{cases}
					Q_{1,l} + Q_{2,l}\Ldfrac{\lambda t}{w}, 
					&\text{if } \Ldfrac{\lambda t}{2w} \leq Q_{3,l}; \\[9pt]
					0, &\text{if } \Ldfrac{\lambda t}{2w} \geq Q_{3,l},
				\end{cases}
			\end{equation*}
			for some rational numbers $Q_{1,l}$, $Q_{2,l}$, and $Q_{3,l}$ with denominators (in their irreducible form) at most~$2h$. Furthermore, we also see that
			\begin{equation*}
				S_{k,r,w}-R_{k,r,w} 
				= \begin{cases}
					\Ldfrac{2}{w} - \Ldfrac{(k+r)t}{w^2}, 
					&\text{if } w\leq \Ldfrac{(k+r+1)t}{2}; \\[9pt]
					\Ldfrac{(k+r+2)t}{w^2} - \Ldfrac{2}{w}, 
					&\text{if } w\geq \Ldfrac{(k+r+1)t}{2}.
				\end{cases}
			\end{equation*}
			Hence, $w(S_{k,r,w}-R_{k,r,w})\sum_{l=1}^{n_{k,r}}\mu(\cI_l)$ is a piece-wise function of $w$ with breaking points at
			\begin{equation*}
				w=\frac{(k+r+1)t}{2} \text{ and } w=\frac{\lambda t}{2Q_l},
			\end{equation*}
			for some rational numbers $Q_l$. So, we partition $[M_{k,r}^-, M_{k,r}^+]$ by
			\begin{equation*}
				M_{k,r}^- = B_0 < B_1 < \cdots < B_N < B_{N+1} = M_{k,r}^+.
			\end{equation*}
			In each interval $[B_j, B_{j+1}]$, we have constants $K_{k,r,j,1}$, $K_{k,r,j,2}$, $K_{k,r,j,3}$ and $K_{k,r,j,4}$ such that
			\begin{equation*}
				w\int_0^1 \mu\big(A_{k,r,w}(x)\big)\,\dv x 
				= K_{k,r,j,1} + K_{k,r,j,2} \frac{t}{w} + K_{k,r,j,3} \frac{\lambda t}{w} + K_{k,r,j,4} \frac{\lambda t^2}{w^2}.
			\end{equation*}
			The partition will depend on $\lambda$. 
			For example, the order of the partition may depend on whether, for some rational number $Q$, 
			is $\lambda t/2Q \leq (k+r+1)t/2$ or $\lambda t/2Q \geq (k+r+1)t/2$? 
			Then, $G_{t}(\lambda)$ will be a piece-wise function with breaking points at least at rational numbers with denominators no more than~$2h$. Furthermore, the restriction $M_{k,r}^-\leq M_{k,r}^+$ implies that $\lambda \leq 2/t$. (Since this is true for all~$k,r$, we see that $G_t(\lambda) = 0$ for all $\lambda\geq 2/t$.) Also, the definition of $M_{k,r}^-$ as the minimum of $\lambda t/2$ and $(k+r)t/2$ will also give us more breaking points at rational numbers. Furthermore, in the cases when $M_{k,r}^+=1$, our partition will also depend on whether $\lambda t/2Q \leq 1$ or $\lambda t/2Q \geq 1$, so we get breaking points at $\lambda = 2Q/t$. All in all, we have the following partition
			\begin{equation*}
				0 = \Lambda_0(t) < \Lambda_1(t) < \cdots < \Lambda_M(t) < \Lambda_{M+1}(t) = 2/t,
			\end{equation*}
			such that each $\Lambda_i(t)$ can be written as either
			\begin{equation*}
				\Lambda_i(t) = Q_i,
			\end{equation*}
			where $Q_i$ is a rational number no greater than $2/t$, with denominator no greater than $2h$, or
			\begin{equation*}
				\Lambda_i(t) = \frac{2}{t} Q_i,
			\end{equation*}
			where $Q_i$ is a rational less than or equal to $1$, with denominator no greater than $2h$. 
			This partition depends on $t$, so we need to partition the interval $(2/(h+1), 2/h]$ into
			\begin{equation*}
				\frac{2}{h+1} = T_{h,0} < T_{h,1} < \cdots < T_{h, K_h} < T_{h, K_h+1} = \frac{2}{h},
			\end{equation*}
			such that throughout each interval $[T_{h, k}, T_{h, k+1}]$, the sequence $\Lambda_i(t)$ remains in the same order. Each $T_{h,k}$ will be a rational number with denominator no bigger than $2h(h+1)$.
			Assume that $\lambda$ lies in some interval $[\Lambda_i(t), \Lambda_{i+1}(t)]$ such that the partition $B_0 < \cdots < B_{N+1}$ is fixed. Each $B_j$ is of the form either
			\begin{equation*}
				B_j = \alpha_{k,r,j} \text{ or } \alpha_{k,r,j}t \text{ or } \alpha_{k,r,j}\lambda t.
			\end{equation*}
			Therefore, we have
			\begin{align*}
				\int_{B_j}^{B_{j+1}}\bigg( K_{k,r,j,1} + K_{k,r,j,2} \frac{t}{w} + K_{k,r,j,3} \frac{\lambda t}{w} + K_{k,r,j,4} \frac{\lambda t^2}{w^2}\bigg) \,\dv w &= \beta_{k,r,j,1} + \beta_{k,r,j,2} t + \beta_{k,r,j,3} \lambda t + \beta_{k,r,j,4} \lambda t^2 \\
				&\phantom{=}+\beta_{k,r,j,5}t\log(\lambda) + \beta_{k,r,j,6}\lambda t\log(\lambda) \\
				&\phantom{=}+\beta_{k,r,j,7}t\log(t) + \beta_{k,r,j,8}\lambda t\log(t),
			\end{align*}
			for some constants $\beta_{k,r,j,1}$, $\beta_{k,r,j,2}$, $\beta_{k,r,j,3}$, $\beta_{k,r,j,4}$, $\beta_{k,r,j,5}$, $\beta_{k,r,j,6}$, $\beta_{k,r,j,7}$, and $\beta_{k,r,j,8}$. Consequently,
			\begin{align*}
				\int_{M_{k,r}^-}^{M_{k,r}^+} w\int_0^1 
				\mu\big(A_{k,r,w}(x)\big)\,\dv x \,\dv w 
				&= \beta'_{k,r,1} + \beta'_{k,r,2} t + \beta'_{k,r,3} \lambda t + \beta'_{k,r,4} \lambda t^2 \\
				&\phantom{=}+\beta'_{k,r,5}t\log(\lambda) + \beta'_{k,r,6}\lambda t\log(\lambda) \\
				&\phantom{=}+\beta'_{k,r,7}t\log(t) + \beta'_{k,r,8}\lambda t\log(t).
			\end{align*}
			For any $t\in[T_{h,j}, T_{h,j+1}]\subseteq[2/(h+1), 2/h]$, and any $\lambda$ in our interval $[\Lambda_i(t), \Lambda_{i+1}(t)]$, the $k,r$ over which we sum are fixed. Hence, there are constants $\kappa_{h,j,i,1}$, $\kappa_{h,j,i,2}$, $\kappa_{h,j,i,3}$, $\kappa_{h,j,i,4}$, $\kappa_{h,j,i,5}$, $\kappa_{h,j,i,6}$, $\kappa_{h,j,i,7}$, and $\kappa_{h,j,i,8}$ such that
			\begin{align*}
				G_t(\lambda) &= \kappa_{h,j,i,1} + \kappa_{h,j,i,2}t + \kappa_{h,j,i,3}\lambda t + \kappa_{h,j,i,4}\lambda t^2 \\
				&\phantom{=}+ \kappa_{h,j,i,5}t\log(\lambda) + \kappa_{h,j,i,6}\lambda t\log(\lambda) \\
				&\phantom{=}+ \kappa_{h,j,i,7}t\log(t) + \kappa_{h,j,i,8}\lambda t\log(t).
			\end{align*}
			This completes the description of the gap distribution function formula and concludes the proof of Theorem~\ref{T:GLambdaAlltConstants}.


		\end{document}